\documentclass[11pt]{amsart}
  \usepackage{amsopn, amssymb, amscd, amsmath, amsthm}
\usepackage{hyperref}
\usepackage{graphicx, graphics, epsfig}
\usepackage{enumerate}
\usepackage{xcolor}
\usepackage{mathdots}

\usepackage{lscape, pdflscape}  %para escribir en vertical
\usepackage{array, multirow} %multirows en la tabla
\allowdisplaybreaks

\newcommand{\raa}{{\rm (a)}}
\newcommand{\rbb}{{\rm (b)}}
\newcommand{\rcc}{{\rm (c)}}

\newcommand{\ri}{{\rm (i)}}
\newcommand{\rii}{{\rm (ii)}}
\newcommand{\riii}{{\rm (iii)}}

\newcommand{\J}{\mathcal{J} }
\newcommand{\nc}{\newcommand}
\nc{\orange}{\textcolor{orange}}

\nc{\Uca}{\mathcal{U}}

\nc{\RS}{\sigma_{\operatorname{Ric}}} \nc{\REV}{\operatorname{RicEV}}

\nc{\pr}{\operatorname{pr}}

\nc{\Gv}{{\G_{(\vg_i)} } }   \nc{\ggov}{{\ggo_{(\vg_i)} }}   \nc{\Gvt}{{\G^t_{(\vg_i)} } } 
\nc{\ig}{\mathfrak{i}}
\nc{\graf}{\text{graph}}

\nc{\Gl}{\mathsf{GL}} \nc{\Or}{\mathsf{O}}  \nc{\SO}{\mathsf{SO}}   \nc{\Sl}{\mathsf{SL}}
\nc{\G}{\mathsf{G}} \nc{\K}{\mathsf{K}}  \nc{\T}{\mathsf{T}} \nc{\Lsf}{\mathsf{L}}
\nc{\Qb}{\mathsf{Q}_\Beta} \nc{\Hb}{\mathsf{H}_\Beta} \nc{\Ub}{\mathsf{U}_\Beta}
\nc{\Gb}{\mathsf{G}_\Beta} \nc{\Kb}{\mathsf{K}_\Beta}
\nc{\PPP}{\mathsf{P}} \nc{\U}{\mathsf{U}} \nc{\N}{\mathsf{N}} \nc{\Ss}{\mathsf{S}} \nc{\Aa}{\mathsf{A}}
\nc{\Hh}{\mathsf{H}}

\nc{\la}{\langle} \nc{\ra}{\rangle}

\nc{\laH}{\la\!\la} \nc{\raH}{\ra\!\ra}

\nc{\ipH}{{\laH \cdot, \cdot \raH}}

\nc{\iph}{{\la h \cdot , h \cdot \ra}}

\nc{\Vg}{{V(\ggo)}}

\nc{\alert}{\color{blue}}
\nc{\fg}{\mathfrak{f}}  \nc{\vg}{\mathfrak{v}} \nc{\wg}{\mathfrak{w}} \nc{\zg}{\mathfrak{z}} \nc{\ngo}{\mathfrak{n}} \nc{\kg}{\mathfrak{k}} \nc{\mg}{\mathfrak{m}} \nc{\bg}{\mathfrak{b}} \nc{\ggo}{\mathfrak{g}} \nc{\ggob}{\overline{\mathfrak{g}}} \nc{\sog}{\mathfrak{so}} \nc{\sug}{\mathfrak{su}} \nc{\spg}{\mathfrak{sp}} \nc{\slg}{\mathfrak{sl}} \nc{\glg}{\mathfrak{gl}} \nc{\cg}{\mathfrak{c}} \nc{\rg}{\mathfrak{r}}  \nc{\hg}{\mathfrak{h}} \nc{\tgo}{\mathfrak{t}} \nc{\ug}{\mathfrak{u}} \nc{\dg}{\mathfrak{d}} \nc{\ag}{\mathfrak{a}} \nc{\pg}{\mathfrak{p}} \nc{\sg}{\mathfrak{s}} \nc{\affg}{\mathfrak{aff}} \nc{\qg}{\mathfrak{q}} \nc{\ugo}{\mathfrak{u}}
\nc{\Xg}{\mathfrak{X}} \nc{\lgo}{\mathfrak{l}}

\nc{\pca}{\mathcal{P}} \nc{\nca}{\mathcal{N}} \nc{\lca}{\mathcal{L}} \nc{\oca}{\mathcal{O}} \nc{\mca}{\mathcal{M}} \nc{\tca}{\mathcal{T}} \nc{\aca}{\mathcal{A}} \nc{\cca}{\mathcal{C}} \nc{\gca}{\mathcal{G}} \nc{\sca}{\mathcal{S}} \nc{\hca}{\mathcal{H}} \nc{\bca}{\mathcal{B}} \nc{\dca}{\mathcal{D}}

\nc{\vp}{\varphi} \nc{\ddt}{\tfrac{{\rm d}}{{\rm d}t}} \nc{\dds}{\tfrac{{\rm d}}{{\rm d}s}} \nc{\ddtbig}{\frac{{\rm d}}{{\rm d}t}} \nc{\dd}{{\rm d}}
\nc{\dpar}{\tfrac{\partial}{\partial t}} \nc{\im}{\mathtt{i}}
\renewcommand{\Im}{{\rm Im}}

 \nc{\SU}{\mathsf{SU}} 
\nc{\RR}{{\mathbb R}} \nc{\HH}{{\mathbb H}} \nc{\CC}{{\mathbb C}} \nc{\ZZ}{{\mathbb Z}}
\nc{\FF}{{\mathbb F}} \nc{\NN}{{\mathbb N}} \nc{\QQ}{{\mathbb Q}} \nc{\PP}{{\mathbb P}}

\nc{\vs}{\vspace{.2cm}} \nc{\vsp}{\vspace{1cm}} \nc{\ip}{{\langle\cdot,\cdot\rangle}}
\nc{\ipp}{(\cdot,\cdot)} \nc{\unm}{\tfrac{1}{2}}
\nc{\unc}{\tfrac{1}{4}} \nc{\und}{\tfrac{1}{16}} \nc{\no}{\vs\noindent}
\nc{\lam}{\Lambda^2(\RR^n)^*\otimes\RR^n} \nc{\tangz}{{\rm T}^{\rm Zar}}
\nc{\lamg}{\Lambda^2\ggo^*\otimes\ggo}
\nc{\nor}{{\sf n}}  \nc{\mum}{/\!\!/} \nc{\kir}{/\!\!/\!\!/}
\nc{\Ri}{\tfrac{4\Ric_{\mu}}{||\mu||^2}} \nc{\ds}{\displaystyle}
\nc{\lb}{[\cdot,\cdot]} \nc{\isn}{\tfrac{1}{||v||^2}}
\nc{\gkp}{(\ggo=\kg\oplus\pg,\ip)} \nc{\ukh}{(\ug=\kg\oplus\hg,\ip)}
\nc{\tgkp}{(\tilde{\ggo}=\kg\oplus\pg,\ip)}
\nc{\wt}{\widetilde}
\nc{\raw}{\rightarrow} \nc{\lraw}{\longrightarrow} \nc{\hqn}{\mathcal{H}_{q,n}}

\nc{\Spec}{\operatorname{Spec}}

\nc{\ad}{\operatorname{ad}}  \nc{\Aut}{\operatorname{Aut}}   \nc{\Inn}{\operatorname{Inn}}   \nc{\Lie}{\operatorname{Lie}} \nc{\Ad}{\operatorname{Ad}} \nc{\Der}{\operatorname{Der}} \nc{\rad}{\operatorname{rad}} \nc{\kf}{\operatorname{B}}

\nc{\End}{\operatorname{End}} \nc{\rank}{\operatorname{rank}} \nc{\Ker}{\operatorname{Ker}} \nc{\tr}{\operatorname{tr}} \nc{\sym}{\operatorname{sym}} \nc{\diag}{\operatorname{diag}} \nc{\proy}{\operatorname{pr}} \nc{\Adj}{\operatorname{Adj}} \nc{\vspan}{\operatorname{span}}

\nc{\Hess}{\operatorname{Hess}}  \nc{\dif}{\operatorname{d}} \nc{\sen}{\operatorname{sen}} \nc{\grad}{\operatorname{grad}} \nc{\Order}{\operatorname{O}} \nc{\divg}{\operatorname{div}}

\nc{\Iso}{\operatorname{Iso}} \nc{\Diff}{\operatorname{Diff}} \nc{\Rc}{\operatorname{Rc}} \nc{\Ricci}{\operatorname{Ric}} \nc{\Riem}{\operatorname{Rm}} \nc{\scalar}{\operatorname{sc}} \nc{\scalarm}{\hat{\operatorname{R}}} \nc{\Riccim}{\widehat{\operatorname{Ric}}} \nc{\tang}{\operatorname{T}} \nc{\vol}{\operatorname{vol}}

\nc{\mm}{\operatorname{M}} \nc{\CH}{\operatorname{CH}} \nc{\Irr}{\operatorname{Irr}} \nc{\mcc}{\operatorname{mcc}} \nc{\m}{\operatorname{m}}

\nc{\Id}{\operatorname{Id}}  \nc{\mmm}{\operatorname{m}}
\nc{\id}{\operatorname{id}} 

\theoremstyle{plain}
\newtheorem{theorem}{Theorem}[section]
\newtheorem{proposition}[theorem]{Proposition}
\newtheorem{corollary}[theorem]{Corollary}
\newtheorem{lemma}[theorem]{Lemma}

\theoremstyle{definition}

\theoremstyle{plain}

\theoremstyle{remark}
\newtheorem{remark}[theorem]{Remark}

\newtheorem{example}[theorem]{Example}

\title[Complex and symplectic structures on almost abelian Lie groups]{Classification of  almost abelian Lie groups   admitting left-invariant  complex or symplectic~structures}

\author[R. M. Arroyo]{Romina M. Arroyo} \thanks{R. M. Arroyo, M. L. Barberis and Y. Godoy were  partially supported by CONICET and SECyT - UNC (Argentina). R. M. Arroyo was also partially supported by FONCYT (Argentina)} 
\address{CIEM - CONICET, FAMAF - Universidad Nacional de Córdoba, Ciudad Universitaria, 5000 Córdoba, Argentina}
\email{romina.arroyo@unc.edu.ar} \email{mlbarberis@unc.edu.ar}
\email{yamile.godoy@unc.edu.ar}
\author[M. L. Barberis]{María L. Barberis} 
\author[V. S. Diaz]{Verónica S. Diaz} 
\address{CEMIM - FCEyN - Universidad Nacional de Mar del Plata, Funes 3350, 7600 Mar del Plata, Argentina}
\email{veronicadiaz@mdp.edu.ar}
\thanks{V.S. Diaz was partially supported by ANPCyT (Argentina)}
\author[Y. Godoy]{Yamile Godoy} 
\author[I. Hernández]{Isabel Hern\'andez}
\address{SECIHTI - CIMAT Unidad M\'erida, M\'exico}
\email{isabel@cimat.mx}
\thanks{I. Hern\'andez was partially supported by grant CONAHCYT  A1-S-45886 (M\'exico)}

\thanks{The present project emerged at the ``Latin American and Caribbean Workshop on Mathematics and Gender” held at Casa Matemática Oaxaca (CMO),  May 15 - 20th, 2022.
}

\keywords{Almost abelian Lie algebra, complex structure, symplectic structure}
\subjclass[2020]{53D05, 32M10,  22E25, 53C30, 17B05, 17B30}

\begin{document}
\begin{abstract} {We classify the  almost abelian Lie algebras $\ggo_A=\RR e_0 \ltimes_A  \RR^{2n-1}$ admitting complex or symplectic structures. The matrix $A\in M(2n-1,\RR )$ encodes the adjoint action of $e_0$ on the abelian ideal $\RR^{2n-1}$, and the existence of complex or symplectic structures on $\ggo_A$ imposes restrictions on the Jordan  normal form of $A$. The classification essentially reduces to the case when $A$ is nilpotent, so we start by considering this case.  
It turns out that if $A$ is nilpotent and $\ggo_A$ admits a complex structure, then $\ggo_A$ necessarily admits a symplectic structure. This is not 
 true in  general when $A$ is non-nilpotent.  
 Finally, several consequences of the classification theorems are obtained.}
\end{abstract}

\maketitle

\section{Introduction}

Nilmanifolds, and more generally, solvmanifolds admitting geometric structures play an important role in differential geometry. They have provided examples and counterexamples to open questions and  conjectures and have been extensively studied by several authors (see \cite{AV, AB, AN, Baz,  BFLT, BGM, CG, CM, CT, ChH, dB0, FP1, GR, LZ, NW, OOS, U}, among many others). Most of the first known examples of compact symplectic manifolds which do not admit any K\"ahler structure are nilmanifolds (see, for  instance, \cite{CFG,Th}).  It is well known that the only nilmanifolds admitting K\"ahler structures are tori \cite[Theorem A]{BG}.

The study of left-invariant complex or symplectic structures on solvable Lie groups is an active field of research and several results are known in low dimensions (see for instance, \cite{BG, Ca, CM, CT, CG, COUV16, CFG, dB1, dB2, DT, FGG, FP1, GB, GJK, GZZ, LUV19, LUV23, OS, Ov, Ov2, Rol, Sal}). The classification  of  solvable  Lie groups admitting left-invariant complex or symplectic structures is an interesting problem (see for instance  \cite{Gua}). 
Although there are many contributions to the subject, it is considered a wild problem and it is far from being solved in general.

In this article we focus on a family of solvable Lie groups, namely, the  almost abelian ones.  A Lie group $G$ is called {almost abelian} if its Lie algebra $\ggo$ has a codimension-one abelian ideal $\hg$, that is,  given $e_0\notin \hg, $ $\ggo$ decomposes as  $\ggo=\RR e_0\ltimes \hg$. Such a Lie algebra will be also called almost abelian. Fix a basis $\mathcal B$ of $\hg$ and let  $A$ be the matrix of $\ad_{e_0}|_{\hg}$ with respect to~$\mathcal B$. 
If $\dim\hg =d$, then $\ggo$ is isomorphic to  $\ggo_A:=\RR e_0 \ltimes_A  \RR^d$, with    $[e_0,v]=Av$, where we consider $v\in \RR^d$ as a column vector.  Accordingly, the corresponding simply connected Lie group $G_A$ is a semidirect product 
$G_A=\RR\ltimes_\phi \RR^d$,    where  $\RR$ acts on $\RR^d$   by $\phi(t)=e^{t A}$, that is, group multiplication  in $G_A$ is given by:
\[ (t,v)(s,w)=(t+s, v+e^{t A}w), \qquad t,s \in \RR, \; v,w\in \RR ^d.
\]
The almost abelian Lie algebra  $\ggo_A$ is nilpotent if and only 
if $A$ is nilpotent. 

\medskip

 In order to prove the classification theorems we need to obtain the conjugacy classes of certain $(2n-1)\times (2n-1)$  matrices which characterize the almost abelian Lie algebras admitting complex or symplectic structures. 

We start by considering the nilpotent case (see Sections  \S\ref{sec:nilp-complex} and \S\ref{sec:nilp-symp}).   Isomorphism classes of $2n$-dimensional nilpotent almost abelian Lie algebras are parametrized by conjugacy classes of $(2n-1)\times (2n-1)$ nilpotent  matrices. These  are in  one-to-one correspondence   with partitions of $2n-1$, due to the Jordan normal form of nilpotent matrices. In our classification results, Theorems~\ref{classifcomplex} and \ref{classifsymp}, we determine which partitions of $2n-1$ correspond to $2n$-dimensional nilpotent almost abelian  Lie algebras admitting complex or symplectic structures.  As a consequence of Theorem  \ref{classifcomplex} and Theorem \ref{classifsymp}  we obtain Theorem \ref{complex-sympl}, which asserts  that if $\ggo_A$ admits a complex structure, then it necessarily admits a symplectic structure. This does not hold for arbitrary nilpotent Lie algebras (see Example \ref{ex_symp-no-cx}). 

The general case follows by applying suitably the results of the nilpotent case. 
The classification is carried out  by making use of the description of conjugacy classes of matrices in  $M_{J_m}(2m, \RR)$  or $\mathfrak{sp}(2m,\RR)$ (see Proposition \ref{solv-complex2n} and Theorem \ref{prop:symplex-nonzero}, respectively). The main results are Theorem \ref{complex_M(a)+Q} and Theorem\ref{sympM(c)+M(-c)+Q}, and from these we obtain the classification in the complex and symplectic case, respectively.

Several consequences of our main results are obtained throughout the paper.

\ 

{\it Acknowledgements.}  The authors are grateful to Leandro Cagliero for insightful discussions on nilpotent orbits in the symplectic Lie algebra, and to Jorge Lauret for suggesting to explore the validity of Theorem \ref{teo:nilpotent-part}. We owe special thanks to Marco Freibert for drawing our attention to reference \cite{BFLT} and for encouraging us to extend our results to the solvable case, which led to significant improvements in the article.

R. M. Arroyo would like to acknowledge support from the ICTP through the Associates Programme (2023-2028).

\section{Preliminaries}

Let $\ggo$ be a real Lie algebra.  A complex structure on $\ggo$ is an automorphism
$j$ of $\ggo$ satisfying $j^2=-\id$ and the 
integrability condition $N_{j}(x,y) =0$ for all $x,y$ in $\ggo$, where $N_j$ is the Nijenhuis tensor:
$$
N_{j}(x,y) =
[x,y]+j([jx,y]+[x,jy])-[jx,jy].  
$$
A symplectic structure on $\ggo$ is a non-degenerate $2$-form on $\ggo$ which is closed, that is,
$d\omega =0$, where 
\[ d\omega (x,y,z)= -\omega([x,y],z)-\omega([y,z],x)-\omega([z,x],y), \qquad\quad x,y,z \in \ggo .
\]
If $G$ is a connected Lie group with Lie algebra $\ggo$, any complex structure on $\ggo$ induces a left-invariant complex structure on $G$. In this way, $G$ becomes a complex manifold such that left translations are holomorphic diffeomorphisms. Similarly, a symplectic structure on $\ggo$ induces a left-invariant symplectic structure on $G$.

Recall that a Lie group $G$ is called {almost abelian} if its Lie algebra $\ggo$ has a codimension-one abelian ideal $\hg$, that is,  given $e_0\notin \hg, $ $\ggo$ decomposes as  $\ggo=\RR e_0\ltimes \hg$. Such a Lie algebra will be also called almost abelian. Fix a basis $\mathcal B$ of $\hg$ and let  $A$ be the matrix of $\ad_{e_0}|_{\hg}$ with respect to~$\mathcal B$. 
If $\dim\hg =d$, then $\ggo$ is isomorphic to  $\ggo_A:=\RR e_0 \ltimes_A  \RR^d$, with    $[e_0,v]=Av$, where we consider $v\in \RR^d$ as a column vector.   Note that the Jacobi identity is automatically satisfied.  
Accordingly, the corresponding simply connected Lie group $G_A$ is a semidirect product 
$G_A =\RR\ltimes_\phi \RR^d$,    where  $\RR$ acts on $\RR^d$   by $\phi(t)=e^{t A}$, that is, group multiplication  in $G_A$ is given by:
\[ (t,v)(s,w)=(t+s, v+e^{t A}w), \qquad t,s \in \RR, \; v,w\in \RR ^d.
\]
The almost abelian Lie algebra  $\ggo_A$ is nilpotent if  and only if $A$ is nilpotent.

Regarding the isomorphism classes of almost abelian Lie algebras, we have the following result, due to Freibert \cite{Fr}.
\begin{proposition}
 \cite[Proposition 1]{Fr}\label{ad-conjugated}
Given $A, B\in M(d,\RR)$, the almost abelian Lie algebras $\ggo_A$ and $\ggo_B$ are isomorphic if and only if there exists $c\neq 0$ such that $B$ and $cA$ are conjugate. 
\end{proposition}

\begin{remark}\label{rem:nilp-similiar}
It follows that two nilpotent almost abelian  Lie algebras $\ggo_A$ and $\ggo_B$ are isomorphic if and only if $A$ and $B$ are conjugate, since for   any   nilpotent matrix $N$, $cN$ and $N$ are conjugate  for $c\neq 0$. Therefore, isomorphism classes of $(d+1)$-dimensional nilpotent almost abelian Lie algebras are parametrized by conjugacy classes of $d\times d$ nilpotent  matrices. 
\end{remark}

\section{Useful linear algebra results}

Given two matrices $M_i\in M(d_i, \mathbb R)$, for $i=1,2$, we denote by $M_1\oplus M_2$ the matrix $\begin{pmatrix}
    M_1&0 \\ 0&M_2
\end{pmatrix}\in M(d_1 + d_2, \mathbb R)$ and we use a similar notation for 3 or more matrices. In particular, given a matrix $M$, we write $M^{\oplus p}$ for the matrix $M\oplus \cdots \oplus M$ ($p$ copies). 
Finally, $0_s$ denotes the null $s\times s$ matrix and  
$I_t$ denotes the identity matrix of size $t\times t$.

Let $C\in M(r,\RR)$, we denote by $\text{spec}\,(C)$  the set of distinct eigenvalues of $C$ and by $\text{spec}_\RR(C)$ 
the set of distinct real eigenvalues of $C$. For $a\in \RR$ and $\alpha \in \CC$ with  Im$(\alpha)>0$, we consider the following sets of matrices:
\begin{eqnarray*}
    M_a(r,\RR)&:=&\{C\in M(r,\RR) : \, \text{spec}\, (C)= \{a\}\}, \\ 
    M_\alpha ( r,\RR)&:=&\{C\in M(r,\RR) : \, \text{spec}\, (C)= \{\alpha , \overline{\alpha}\}\}.
\end{eqnarray*}
Note that when $\alpha \in \CC$ with Im$(\alpha)>0$, $r$ is necessarily even.

Given $\lambda\in \CC$ and $n\geq 2$, $\J_{n}(\lambda)$ denotes the Jordan block of size $n\times n$ :
\begin{equation*}
\J_{n}(\lambda)= \begin{pmatrix}
  \lambda & 0 & \cdots & \cdots & 0\\
  1 & \lambda &  &  &  0 \\
  0 & \ddots & \ddots &  & \vdots \\
  \vdots & \ddots & \ddots & \ddots &  \vdots \\
  0 & \cdots & 0 & 1 &\lambda
\end{pmatrix}.
\end{equation*} When $\lambda=0$, we simply denote $\J_n$ instead of $\J_n(0)$.

For each  $m\in\NN$, we denote by $\bigwedge_m$ the following set of tuples of integers: 
\begin{equation*}    \begin{split}
    {\bigwedge}_m:&=\{ (n_1,\ldots , n_k; p_1, \ldots , p_k; t ): \sum_i n_ip_i +t= m,  \\   & \qquad 
k\geq 1, \;  n_1> \dots > n_k\geq 2, \;  p_i > 0 \; \forall i, \; t\geq 0  \}\cup\{(m)\}.
 \end{split}\end{equation*}

\vspace{0.5cm}

 Let $A\in M_b(m,\RR)$,  $b\in \RR$. Then $A$ is conjugate to either 
\begin{equation}\label{real eigenvalue}
\begin{cases} \J_{n_1}(b)^{\oplus p_1}\oplus \dots \oplus \J_{n_k}(b)^{\oplus p_k} \oplus bI_t,   \quad (n_1,\ldots , n_k; p_1, \ldots , p_k; t )\in {\bigwedge}_m, \\
\text{or}\\
bI_m.
\end{cases}
 \end{equation} 
 We associate to $A$ the following tuple $\Lambda(A)$:
\begin{equation}\label{Lambda}
\Lambda(A)=(n_1,\ldots , n_k; p_1, \ldots , p_k; t ), \quad \text{ or } \quad  \Lambda(A)=(m), \text{ respectively.}
\end{equation}  

Given $b,d\in \RR$ where $d> 0$, let 
\begin{equation}\label{eq:complex-matrix}
    \mathcal{C}(b,d)=\begin{pmatrix}
b & d \\
-d & b
\end{pmatrix}.
\end{equation}
For $n\geq 2$ we define 
\begin{equation}\label{eq:matriz-inflada}
\mathcal{C}_n (b,d)=\begin{pmatrix}
  \mathcal{C}(b,d) &  &  &  & \\
  I_2 & \mathcal{C}(b,d) &  &  &   \\
   & I_2 & \ddots &  &  \\
   &  & \ddots & \ddots &   \\
   &  &  & I_2 &\mathcal{C}(b,d)
\end{pmatrix} \in M (2n,\RR) .      
\end{equation}
Let $\alpha=b+id$  with $d>0$ and  let $A\in M_\alpha(2m,\RR )$. Then $A$ is conjugate to  either 
\begin{equation}\label{complex eigenvalue}
     \begin{cases}\mathcal{C}_{n_1}(b,d)^{\oplus p_1}\oplus \dots \oplus \mathcal{C}_{n_k}(b,d)^{\oplus p_k} \oplus \mathcal{C}(b,d)^{\oplus t},   \quad (n_1,\ldots , n_k; p_1, \ldots , p_k; t )\in {\bigwedge}_m, \\
     \text{or}\\
     \mathcal{C}(b,d)^{\oplus m}.
    \end{cases}
\end{equation}
We associate to $A$ the following tuple $\Lambda(A)$: 
\begin{equation}
\Lambda(A)=(n_1,\ldots , n_k; p_1, \ldots , p_k; t ), \quad \text{ or } \quad  \Lambda(A)=(m), \text{ respectively.}
\end{equation}

Let $C\in M(r, \RR)$ such that spec$\, (C)=\{\lambda_1 , \ldots , \lambda _{\ell}, \lambda _{\ell +1}, \overline{\lambda} _{\ell +1}, \ldots, \lambda _s, \overline{\lambda }_s\}$, where $\lambda_j\in \RR$, for $j=1,\ldots,\ell$, and $\lambda_j \in \CC$  with $\text{Im}(\lambda_j)>0$, for $j=\ell +1,\ldots,s$. 

Then $C$ is conjugate to  
\begin{equation}\label{Cdecomposition}
C(\lambda_1)\oplus \cdots \oplus C(\lambda_{s}).     
\end{equation}
where $C(\lambda_j)\in M_{\lambda_j}(m_j,\RR)$, $m_1 + \cdots + m_s = r$ and
\begin{itemize}
\item if $j=1,\ldots,\ell$, then $C(\lambda_j)$ is one of the matrices given by \eqref{real eigenvalue}  with $b=\lambda_j$ and $m=m_j$,  
\item if $j=\ell+1,\ldots,s$, the matrix $C(\lambda_j)$ is one of the matrices given by \eqref{complex eigenvalue} with $\lambda_j=b_j+id_j$ and $2m=m_j$.\end{itemize}  
Through the paper,  given a real square  matrix $M$ and $\lambda\in \text{spec}_\RR (M)$, we denote $m_\lambda$ the positive integer satisfying 
 $M(\lambda)\in M(m_\lambda , \RR)$.

\smallskip

The next lemma is a straightforward consequence of \eqref{Cdecomposition}.

\begin{lemma}\label{M(a)+Q}
Given $C\in M(m, \RR)$ and $b \in \text{spec}_\RR(C)$, then $C$ is conjugate to either $C(b)$ or $C(b)\oplus Q$, where 
$Q\in M(m-m_b, \RR )$ and $b\notin \text{spec}\, (Q)$, depending on whether  $m_b=m$ or $m_b<m$.    
\end{lemma}

\begin{remark}\label{L(A)=L(N)}
Observe that if $b\in \RR$ and $C\in M_b(m,\RR )$,   then $\Lambda(C)=\Lambda(C-bI_m)$, since $C(b)=(C-bI_m)(0)+bI_m$.     
\end{remark}

Given $b\in \RR$, we denote by $M_b (m,\RR)/GL(m,\RR)$  the set of conjugacy classes of matrices in $M_b (m,\RR)$. 
It follows from 
 uniqueness of the Jordan normal form up to reordering that 
 tuples corresponding to conjugate matrices in $M_b (m,\RR)$ are equal. Therefore, $\Lambda$ as in \eqref{Lambda} induces  a well defined map   $\tilde{\Lambda} $: 
\begin{equation}\label{eq:bijection-tuples}    
\tilde{\Lambda} : M_b(m,\RR)/GL(m,\RR)\longrightarrow  {\bigwedge}_m, \qquad \quad [A]\mapsto \Lambda(A),\end{equation}
where  $[A]$ denotes the conjugacy class of $A\in M_b(m,\RR)$. The  map $\tilde{\Lambda}$  is clearly a bijection between $M_b(m,\RR)/GL(m,\RR)$ and $  {\bigwedge}_m$.

In view of Proposition \ref{ad-conjugated},  in order to study isomorphism classes of almost abelian Lie algebras admitting complex or symplectic structures, we need to determine the conjugacy classes of matrices of the form \eqref{eq:complex-abelian} and \eqref{eq:symp-matrix}. We will consider separately the nilpotent and non-nilpotent case.

\subsection{The nilpotent case}
In view of Remark \ref{rem:nilp-similiar}, we state and prove the following lemma. 

\begin{lemma}\label{lem:isom-nilpot}Let $m \in \NN$. 
 The isomorphism classes of $(m+1)$-dimensional nilpotent almost abelian  Lie algebras are parametrized by the set $\bigwedge_m$.   
\end{lemma}
\begin{proof}
According to Remark \ref{rem:nilp-similiar}, the isomorphism classes of $(m+1)$-dimensional nilpotent almost abelian Lie algebras are in one-to-one correspondence with the set of
conjugacy classes of nilpotent $m\times m$  matrices. This set  is parametrized by   $\bigwedge_m$ via the bijection $\tilde{\Lambda}$ defined in \eqref{eq:bijection-tuples} for $b=0$, and the lemma follows.
\end{proof}

In the following results we study the Jordan normal form of nilpotent matrices of the form $\begin{pmatrix}
      0   & 0 \\
      v   & B
     \end{pmatrix}\in M(m+1, \RR)$, where $v\in \RR^m$ and $B\in  M(m, \RR)$.

\begin{proposition}\label{(v,B)}
    Let $v \in \mathbb R ^m$ and let     $C= \begin{pmatrix}
      0   & 0 \\
      v   & B
     \end{pmatrix}\in M(m+1, \RR)$, where 
     \[ B=\J _{m_1}  \oplus \cdots \oplus \J _{m_r} \oplus  0_s,\quad \text{ with }\; m_1 \geq \cdots \geq m_r \geq 2, \; s\geq 0. \]
     Then $C$ is conjugate to one of the following matrices:
    \begin{enumerate}
        \item[$\raa$] $0_1 \oplus B$,
        \item[$\rbb$] $\J _{m_1}  \oplus \cdots \oplus \J_{m _{d-1}} \oplus \J_{m_{d} + 1}\oplus \J_{m_{d +1}} \oplus \cdots \oplus \J _{m_r} \oplus  0_s$, for $d \in \{ 1, \dots, r\}$,
        \item[$\rcc$] $\J _{m_1}  \oplus \cdots \oplus \J _{m_r} \oplus \J _{2}\oplus  0_{s-1}$, if $s>0$.
    \end{enumerate}

    \smallskip

    If $B=0_m$ and $v\neq 0$, then $C$ is conjugate to $\J _{2}\oplus  0_{m-1}$.
\end{proposition}

For the proof of the above proposition we need the next lemma.

\begin{lemma}\label{v+w}
   If $B\in M(m, \mathbb R)$ is nilpotent and $w\in \Im\, B$, then 
  $ \begin{pmatrix}
      0   & 0 \\
      v   & B
    \end{pmatrix}\in M(m+1, \mathbb R)$ is conjugate to $\begin{pmatrix}
      0   & 0 \\
      v+w   & B
    \end{pmatrix}$, for all $v \in \mathbb R ^m$.
\end{lemma}
\begin{proof}
    We consider 
    $$
    A=\begin{pmatrix}
     0   & 0 \\
      v   & B
    \end{pmatrix} \quad \text{and}  \quad C=\begin{pmatrix}
      0   & 0 \\  
    v +w  & B
    \end{pmatrix}.
    $$    
    Since $B$ is nilpotent, then $A$ and $C$ are nilpotent. Therefore, 
    $A$ and $C$ are conjugate if and only if  $\dim( \Ker\, A^n)= \dim( \Ker\, C^n)$, for all $n \in \mathbb{N}$. It suffices to show that $\text{rk}\, (A^n)= \text{rk}\, (C^n) $ for all $n\in \mathbb{N}$. 
    
     First, we observe that, for any $n\in \mathbb N$,  
    $$
    A^n=\begin{pmatrix}
      0   & 0 \\
      B^{n-1} v   & B^{n}
    \end{pmatrix} \quad \text{and} \quad C^n=\begin{pmatrix}
      0   & 0 \\
      B^{n-1} v + B^{n} x   & B^{n}
    \end{pmatrix},
    $$
    where $x$ is such that $Bx=w$. Hence, 
    $$
    \text{rk}\, (A^n) =\begin{cases}
    \begin{array}{cc}
        \text{rk}\, (B^n),   & \text{if} \hspace{0.25cm} B^{n-1} v \in \Im\,B^n,   \\
        \text{rk}\, (B^n) + 1,   & \text{if} \hspace{0.25cm} B^{n-1} v \notin \Im\,B^n.  
    \end{array}
      \end{cases}
    $$
    
   On the other hand, since $B^{n-1} v \in  \Im\,B^n$ if and only if  $ B^{n-1} v + B^{n} x \in \Im\,B^n$, we obtain that
 $$
    \text{rk}\,(C^n) =\begin{cases}
    \begin{array}{cc}
        \text{rk}\, (B^n),   & \text{if} \hspace{0.25cm} B^{n-1} v \in \Im\,B^n,   \\
        \text{rk}\, (B^n) + 1,   & \text{if} \hspace{0.25cm} B^{n-1} v \notin \Im\,B^n.  
    \end{array}
      \end{cases}
    $$
 Therefore, $\dim( \Ker\, A^n)= \dim( \Ker\, C^n)$, for all $n$, which implies that $A$ is conjugate to $C$, and the lemma follows. 
\end{proof}

Throughout this article, we use the following notation.  The vector $\epsilon_\ell$ denotes the $\ell$th vector of the canonical basis of $\RR^d$, for $\ell=1, \dots, d$.  

\medskip 

The idea of the proof of the Proposition \ref{(v,B)} is contained in \cite[Theorem 3.4]{AB}.

\begin{proof}[Proof of Proposition \ref{(v,B)}] Set $t= m_1 + \cdots + m_r$ and we rename the elements of the canonical basis of $\mathbb R^{m}$ as $\{ \epsilon^1_1, \dots,\epsilon^1_{m_1}, \dots, \epsilon_1^r, \dots, \epsilon_{m_r}^r, \epsilon_{t+1}, \dots, \epsilon_{t+s}\}$.  We can write $v$ as 
\begin{equation*}
v= \sum_ {i=1}^r v_i \epsilon_1 ^i  + w + w_0,
\end{equation*}
where $v_i \in \mathbb R$, $w\in  \Im\,B$ and $w_0 \in \text{span}\{ \epsilon_{t+1}, \ldots, \epsilon_{t+s}\} \subseteq \Ker\, B$. Note that by Lemma  \ref{v+w} we may assume $w=0$. 
We deal with the following cases:
$$
\begin{cases}
\raa \; v \in \Im\,B ,\\
\rbb \; v \notin  \Im\,B \text{ and } v_i\neq 0 \text{ for some } i\in \{ 1, \dots, r\} ,\\
\rcc \; v \notin  \Im\,B \text{ and } v \in\text{span}\{ \epsilon_{t+1}, \ldots, \epsilon_{t+s}\} .
\end{cases}
$$

(a) If $v\in \Im\,B$, then by Lemma \ref{v+w},  $C$ is conjugate to $0_1 \oplus B$.

\medskip

(b) If $v\notin \Im\,B$ and there exists $i$ such that $v_i\neq 0$, set $d:=\text{min}\{i: v_i\neq 0\}$. Let 
$$
A=\J _{m_1}  \oplus \cdots \oplus \J_{m _{d-1}} \oplus \J_{m_{d} + 1}\oplus \J_{m_{d +1}} \oplus \cdots \oplus \J _{m_r} \oplus  0_s.
$$
In order to prove  that $C$ and $A$ are conjugate, we  show that  $\dim (\Ker\, C^n)=\dim (\Ker \,A^n)$, for $n \in \NN$. Moreover, since
\begin{equation*}
\dim (\Ker \, \J^n_{m_i}) = 
\begin{cases}
m_i    & \text{if } n \geq m_i,  \\
  n   &  \text{if }   n < m_i.
\end{cases}
\end{equation*}
and $C$ and $A$ are both nilpotent matrices such that $C^{m_1 +1}=A^{m_1 +1}=0 $, 
it is enough to show that $\dim (\Ker\, C^n)=\dim (\Ker \,A^n)$, for $1\leq n \leq m_1$.  
Since $v \notin \Im B$, in particular $v\neq 0$, and we have
\begin{equation*}
\dim (\Ker\,C) =\dim (\Ker\,B)=\sum_{i=1}^r \dim (\Ker\,\J_{m_i}) +s= r+s.
\end{equation*}
and
\begin{equation*}
\dim (\Ker\,A)=\sum_{i=1, i\neq d}^r\dim (\Ker\,\J_{m_i}) +\dim (\Ker\,\J_{m_d+1})+s= r+s.    
\end{equation*}
For  $2\leq n \leq m_1$,   let $ \ell  = \max\{ i \,  : m_i \geq n  \}$, i.e., 
$$ \text{either}\quad  
m_1 \geq \dots \geq  m_\ell \geq n > m_{\ell +1} \geq \cdots \geq m_r , \quad \text{or }\;\; m_r\geq n. 
$$
A straightforward computation shows that 
$C^n=\begin{pmatrix}
      0   & 0 \\
      B^{n-1} v   & B^{n}
    \end{pmatrix}
    $  with 
\begin{equation*}
  B^{n-1}v =\sum_ {i=d }^r \, v_i B^{n-1}\epsilon_1 ^i +B^{n-1}w_0 =  \begin{cases}
    \begin{array}{cc}
    0,   & \text{if } n-1\geq m_d \\
      \sum_ {i=d }^r \, v_i \epsilon_n ^i,   & \text{if }  n-1< m_d.
    \end{array}
\end{cases}  
\end{equation*}
Hence, we consider the following cases. 

\noindent {\bf Case: $n-1\geq m_d$.} Notice that in this case $\ell <d$, $\ell <r$ and $\dim(  \Ker C^n) =\dim(\Ker B^n) +1$.  We compute
\begin{eqnarray*}
\dim(\ker B^n)&= &
      \sum_{i=1}^\ell  \dim (\Ker\,\J^n_{m_i}) +
\sum_{i=\ell +1}^r \dim (\Ker\,\J^n_{m_i})+s=n\ell   +  \sum_{i=\ell +1}^r m_i +s. 
\end{eqnarray*}
On the other hand, since $n\geq m_d+1$ and $d > \ell$, it follows that
\begin{eqnarray*}
\dim(\ker A^n)&= &     
 \sum_{i=1}^\ell  \dim (\Ker\,\J^n_{m_i}) +
\sum_
{
\begin{smallmatrix}
i=\ell +1\\
i\neq d
\end{smallmatrix}
}^r \dim (\Ker\,\J^n_{m_i})+  \dim (\Ker\,\J^n_{m_d+1})     +  s , \\
&=&     
n\ell   +  \sum_{i=\ell +1}^r m_i +s+1=  \dim(\Ker B^n) +1=\dim (\Ker C^n )  .
\end{eqnarray*}
\noindent {\bf Case: $n-1<  m_d$.} Notice that, $d\leq \ell$ and $\dim(\ker C^n  )= \dim(\ker B^n  )$. Analogous to the previous case,  \[ \dim(\ker B^n)=\begin{cases}
n\ell   +  \sum_{i=\ell +1}^r m_i +s,  &\text{if }\; \ell<r, \\ nr+s , &\text{if }\; \ell =r. \end{cases}\] 
Since $n< m_d+1$ and $d \leq \ell$, it follows that
\begin{eqnarray*}
\dim(\ker A^n)&= &\begin{cases}
      \sum_
{
\begin{smallmatrix}
i=1\\
i\neq d
\end{smallmatrix}
  }^\ell  \dim (\Ker\,\J^n_{m_i}) +
\sum_
{i=\ell +1}^r \dim (\Ker\,\J^n_{m_i})+  \dim (\Ker\,\J^n_{m_d+1})     +  s, &\text{if }\; \ell<r, \\ \sum_
{
\begin{smallmatrix}
i=1\\
i\neq d
\end{smallmatrix}
  }^r  \dim (\Ker\,\J^n_{m_i}) +  \dim (\Ker\,\J^n_{m_d+1}) +s, &\text{if }\; \ell=r, \end{cases} \\
&=&\begin{cases}
n\ell   +  \sum_{i=\ell +1}^r m_i +s,  &\text{if }\; \ell<r, \\ nr+s , &\text{if }\; \ell =r , \end{cases}\\
&= & \dim (\Ker C^n ).  
\end{eqnarray*}

(c) If $ v \notin  \Im\,B \text{ and } v \in\text{span}\{ \epsilon_{t+1}, \ldots, \epsilon_{t+s}\}$.   Let 
$$A=\J _{m_1}  \oplus \cdots \oplus \J _{m_r} \oplus \J _{2}\oplus  0_{s-1}.$$
In order to show that $C$ and $A$ are conjugate, it is enough to show that  
$\dim( \ker C^{n}) =  \dim (\ker A^{n})$,  for  all $1\leq n\leq m_1-1$, since $C^{m_1}=A^{m_1}=0$. 
For $n=1$, we have that $\dim (\Ker\,C)=\dim (\Ker\,B)= r+s$ and 
$$
\dim (\Ker\,A)= \sum_{i=1}^r \dim (\Ker\,\J_{m_i}) + \dim (\Ker\,\J_{2})+s-1= r+s.
$$
For $2\leq n\leq m_1 -1$,  since $B^{n-1}v=0$, it follows that 
$$
\dim (\Ker\,C^n)= \dim (\Ker\,B^n)+1= \sum_{i=1}^r \dim (\Ker\,\J_{m_i}^n) +s+1.
$$
On the other hand,
$$
\dim (\Ker\,A^n)=\sum_{i=1}^r\dim (\Ker\,\J_{m_i}^n)+\dim (\Ker\,\J^n_{2}) + s-1 = \sum_{i=1}^r\dim (\Ker\,\J_{m_i}^n)+s+1.
$$
Hence, $C$ and  $A$ are conjugate. 

Finally, if $B=0_m$ and $v\neq 0$, let $A=\J _{2}\oplus  0_{m-1}$, then $C^2=A^2=0$ and $\dim (\text{Ker}\, C)= m=\dim (\text{Ker}\,  A)$. Therefore, $C$ is conjugate to $A$.
\end{proof}

\begin{remark} \label{rem:agrupando}
We point out that, in Proposition \ref{(v,B)} $\rbb$, different values of $d$ could  give rise to conjugate 
 matrices. For instance, if $m_i=m$ for all $i$, then the matrices in Proposition \ref{(v,B)} $\rbb$ are conjugate to $\J_{m+1}\oplus \J_m^{\oplus (r-1)} \oplus 0_s$ for all $d$. As a consequence of this simple observation we obtain the next corollary. 
    \end{remark}

\begin{corollary}\label{posibles_C}
Let  $C= \begin{pmatrix}
      0   & 0 \\
      v   & B
      \end{pmatrix}\in M(m+1,\RR)$ be a nilpotent matrix, where $v \in \mathbb R ^m$, $B\in M(m,\RR)$ with $\Lambda(B)=(m_1,\ldots, m_k; q_1,\ldots, q_k;s)$. Then $\Lambda(C)$ is one and only one of the following tuples:
    \begin{enumerate}
        \item[$\raa$] $(m_1,\ldots, m_k; q_1,\ldots, q_k;s+1)$;
        \item[] 
        \item[$\rbb$]  
         $\begin{cases}
            (m_1, \cdots, m_{d-1}, m_d, \cdots,  m_k; q_1, \cdots,  q_{d-1}+1, q_{d}-1, \cdots,q_k; s), & \text{if } m_{d-1}=m_d +1, \\
       \text{ or } &\\
     (m_1, \cdots, m_{d-1}, m_d +1, m_d, \cdots,  m_k; q_1, \cdots,  q_{d-1}, 1, q_{d}-1, \cdots,q_k; s),  & \text{if }m_{d-1}>m_d + 1 > m_{d};
      \end{cases}$
      for $d \in \{ 1, \dots, k\}$;
      \item[] 
        \item[$\rcc$] 
        $(m_1,\ldots, m_k, 2; q_1,\ldots, q_k, 1;s-1)$ for $s>0$.
    \end{enumerate}   

    \smallskip
    
    If $B=0_{m}$ and $v\neq 0$, then $\Lambda(C)=(2;1;m-1)$.
\end{corollary}

\subsection{The non-nilpotent case}

The following results are  analogous  to those  given in the previous section in the non-nilpotent setting. 

\begin{lemma}\label{B(a)+D}
Let $w\in \RR^k,$  $Q \in M(k,\RR)$ and $b\in \RR$ such that $b\notin$ \emph {spec}\,$(Q)$. 
\begin{enumerate}
\item[\emph{(i)}] If $C=\left(\begin{array}{cc}
      b  & 0  \\
      w &  Q 
      \end{array}\right)$, then $C$ is conjugate to $\left(\begin{array}{cc}
      b  & 0  \\
      0 &  Q 
      \end{array}\right)$. 
\item[\emph{(ii)}] 
If $C=\left(\begin{array}{c|c|c}
      b  & 0  & 0\\
      \hline
      u &  R & 0 \\
      \hline
      w & 0 & Q
      \end{array}\right)$, with  $u\in \RR^n$, $R \in M_b(n,\RR)$, then $C$ is conjugate to $\left(\begin{array}{c|c|c}
      b  & 0  & 0\\
      \hline
      u &  R & 0 \\
      \hline
      0 & 0 & Q
      \end{array}\right)$.
      \end{enumerate}
\end{lemma}
\begin{proof}
 Since $Q-bI$ is invertible, let $x=(Q-bI)^{-1} w$.

(i) We observe that $\left(\begin{array}{cc}
      b  & 0  \\
      0 &  Q 
      \end{array}\right)=P^{-1}CP$ with $P=\left(\begin{array}{cc}
      1  & 0  \\
     -x &  I_k 
      \end{array}\right)$.

  (ii) By considering   $P=\left(\begin{array}{c|c|c}
   1   & 0 & 0\\
   \hline
      0   & I_n & 0\\
      \hline
      -x & 0 & I_k
    \end{array}\right)$, it is easy to see that
$\left(\begin{array}{c|c|c}
      b  & 0 & 0\\
      \hline
      u  & R & 0\\
      \hline
      0 & 0 & Q
     \end{array}\right) = P^{-1} C P$.
\end{proof}
Given an almost abelian Lie algebra $\ggo_C$ with $C\neq 0$, it follows from Proposition \ref{ad-conjugated} that the family of Lie algebras isomorphic to $\ggo_C$ is parametrized by the conjugacy classes of $C$ up to a nonzero multiple. 
The purpose of the next lemma is to choose an appropriate matrix $M$  which is conjugate to a multiple of $C$.
\begin{lemma}\label{lem:isom-classes} Let $b\in \RR$, $R\in M_b(n,\RR)$, $u\in \RR^n$ and $Q \in M(k,\RR)$ with $k\geq 1, \; b\notin$ \emph {spec}\,$(Q)$. Consider  \[ C=\left(\begin{array}{c|c|c}
      b  & 0  & 0\\
      \hline
      u &  R & 0 \\
      \hline
      0 & 0 & Q
      \end{array}\right). \]
\begin{enumerate}
    \item[\emph{(i)}]
If  $b\neq  0$, then $\ggo _C$ is isomorphic to $\ggo _M$ where 
   \[ M=\left(\begin{array}{c|c|c}
      1  & 0  & 0\\
      \hline
      \dfrac 1b \, u & \dfrac 1b \,  R & 0 \\
      \hline
      0 & 0 &\dfrac 1b \,  Q
      \end{array}\right) .\]
  \item[\emph{(ii)}]    If  $b= 0$, so that $R$ is nilpotent, and 
  $\emph{spec}\, ( Q)=\{ \lambda_1, \ldots , \lambda_\ell \}$ with $|\lambda _1| \geq \cdots  \geq |\lambda_\ell |>0$, 
  then $\ggo _C$ is isomorphic to $\ggo _M$ where $M$ is given by 
   \[ M=\left(\begin{array}{c|c|c}
      0  & 0  & 0\\
      \hline
      \dfrac 1d \, u & \dfrac 1d \,  R & 0 \\
      \hline
      0 & 0 &\dfrac 1d \,  Q
      \end{array}\right),    \quad \text{where } d=  |\lambda_1 |.\]
      \end{enumerate}
      Moreover, if $\ggo_M \cong \ggo_{M_1}$ with 
      \[M_1 =\left(\begin{array}{c|c|c}
      0  & 0  & 0\\
      \hline
      u_1 &  R_1 & 0 \\
      \hline
      0 & 0 & Q_1
      \end{array}\right), \; R_1 \text{ nilpotent, }  \emph {spec}\, ( Q_1)=\{ \mu_1, \ldots , \mu_s \}, \; 1=|\mu _1| \geq \cdots  \geq |\mu_s |>0,\]
      then $s=\ell$ and $M_1$ is conjugate to $\pm M$.
\end{lemma}

\begin{proof}
    The proof of (i) and the first assertion in (ii) is a straightforward consequence of Proposition \ref{ad-conjugated}.     To complete the proof of (ii), observe that
    \[\text{spec}\, \left( \frac 1d \, Q\right)=\left\{ \frac{\lambda_1}d, \ldots , \frac{\lambda_\ell}d \right\}, \quad \text{with} \quad 1= \frac{|\lambda _1|}d \geq \cdots  \geq \frac{|\lambda_\ell |}d>0 .\] 
Since $\ggo_M$ is isomorphic to $\ggo _{M_1}$, it follows from Proposition \ref{ad-conjugated} that $M_1$ is conjugate to $cM$ for some non-zero real number  $c$. In particular, $\text{spec}\, \left(   Q_1 \right)= \text{spec}\, \left( \dfrac cd \, Q\right)$. Therefore, $s=\ell$ and clearly 
\[ 
\frac{|c|}d\, |\lambda _1| \geq \cdots  \geq \frac{|c|}d \, |\lambda_\ell | ,
\]
therefore, $1=| \mu_1 | = \dfrac {|c|}d \, |\lambda _1|=|c|$,  which implies $c=\pm1$ since $c\in \RR$. This completes the proof of (ii).
\end{proof}

\section{Classification of almost abelian Lie algebras admitting complex structures}

\subsection{Characterization of almost abelian Lie algebras admitting a complex structure} In this section it will be useful to view $M (m, \CC )$ inside of $M (2m , \RR)$ as follows.    Define $\varphi: M (m, \CC ) \to M (2m , \RR)$ by:
\begin{equation}\label{eq:alg-iso}    
 \varphi(M)= 
\begin{pmatrix} M_1 & -M_2 \\
M_2 & M_1
\end{pmatrix},\quad \text{ where } \quad M=M_1+iM_2, \;\; M_1,\, M_2 \in M (m , \RR). 
\end{equation} 
Setting 
\begin{equation}\label{eq:Jm} J_m=
 \begin{pmatrix} 0&-I_m \\
    I_m&0
\end{pmatrix},   
 \end{equation} 
the image of $\varphi$ is the Lie subalgebra of $M(2m, \RR)$ given by
\begin{equation*}
\begin{split}
 M_{J_m}(2m, \RR)& :=  \left\{ B\in M(2m, \RR): B J _m = J _m B\right \} \\
 & = 
 \left\{ \begin{pmatrix} M_1 & -M_2 \\
M_2 & M_1
\end{pmatrix} : M_1, M_2 \in M(m, \RR) \right\}.    
\end{split}
\end{equation*} 
Since $\varphi$ is an injective $\RR$-algebra homomorphism, we obtain that $M(m, \CC )$ and  $M_{J_m}(2m, \RR)$ are isomorphic as real Lie algebras.
\begin{remark}\label{rem:k-step}
We point out that, a matrix $M\in M (m, \CC )$ is nilpotent if and only if  
$\varphi(M) \in M (2m , \RR)$ is nilpotent. Moreover, if $M\in M (m, \CC )$ is nilpotent,  then $M$ and $\varphi(M)$ have the same minimal polynomial. Let $x^k$ be the minimal polynomial of $\varphi(M)$, then $x^k$  is the minimal polynomial of $M$ and therefore $k\leq m$.
\end{remark} 

\begin{remark} \label{rem_complex-conj} In general, each $J\in M(2m, \RR)$ satisfying $J^2= -I_{2m}$ determines the Lie subalgebra of $M(2m, \RR)$ given by 
\begin{equation}\label{M_J}
M_{ J}(2m, \RR):= \left\{B\in M(2m, \RR): B J =  J B\right \},
\end{equation} which is isomorphic to $M(m,\CC)$ as real Lie algebras. Moreover, $M_{ J}(2m, \RR)$ is conjugate to $M_{ J_{m}}(2m, \RR)$ in $GL(2m, \RR)$.
\end{remark}
  
\medskip
 
Given a $2m$-dimensional real vector space $V$ and an endomorphism $j$ of $V$ satisfying $j^2=-\id$, we define
\[ \glg(V , j):=\{ S\in \glg(V) : S\circ j=j \circ S \}. \]

Given a basis $\mathcal B$ of $V$ and $S\in \glg(V , j)$, 
we denote by $B$ and $J$ the matrices of $S$ and $j$ with respect to $\mathcal B$, respectively. Then, using the notation introduced in \eqref{M_J}, $B \in M_J\left(2m, \RR\right)$.

We recall from \cite[Lemma 6.1]{LRV} the characterization of almost abelian Lie algebras with a complex  structure (see also \cite[Lemma 3.1]{AO}):
\begin{proposition}\label{complex} 
Let $\ggo$ be a $2n$-dimensional almost abelian Lie algebra with  codimension-one abelian ideal $\hg$. If $j$ is a complex structure on $\ggo$, then there exist $e_0\notin \hg$  and a $j$-invariant ideal $\mathfrak u \subset \hg$ such that  $j e_0\in \hg$, $\hg =\RR je_0 \oplus \mathfrak u$ and $\ad_{e_0}|_{\mathfrak u} \in \glg(\mathfrak u , j_{\mathfrak u}) $, where $j_{\mathfrak u}=j|_ {\mathfrak u}$. 
\end{proposition}

\begin{remark}\label{rem:complex}
According to Proposition \ref{complex}, given a complex structure $j$ on a $2n$-dimensional almost abelian Lie algebra $\ggo$, there exists a basis $\{e_0, e_1,\ldots,e_{2n-1}\}$ of $\ggo$ such that $e_1=je_0$, $\hg=\text{span}\, \{e_i : i\geq 1\}$ and $\mathfrak u =\text{span}\, \{e_i : i\geq 2\}$. Let $J$ be the matrix of $j_{\mathfrak u}$ in the given basis of $\mathfrak u$, then  
the matrix $C$ of $\ad_{e_0}|_\hg$ in the basis $\{ e_i : i\geq 1\}$ is given by 
\begin{equation}\label{eq:complex-abelian}    
C=\begin{pmatrix} 
      a   & 0 \\
      v   & B
    \end{pmatrix}, \qquad  \text{where }a \in \RR,\; v\in \RR^{2(n-1)}\; \text{and } B\in M_{J}(2(n-1), \RR).
\end{equation}   

Conversely, let $C$ be a matrix  as in \eqref{eq:complex-abelian}  where $J\in   M(2(n-1), \RR)$ and $J^2=-I_{2(n-1)}$. Consider the $2n$-dimensional almost abelian Lie algebra $\ggo_C= \RR e_0 \ltimes _C \RR^{2n-1}  $, that is, $[e_0, e_i]=Ce_i$, $i\geq 1$.   
Let $\mathfrak u =\text{span}\, \{e_i : i\geq 2\}$ and define $j\in \glg(\ggo_C)$ such that the matrix of  $j|_{\mathfrak u}\in 
 \glg(\mathfrak u)$  in this basis is $J$ and $j e_0=e_1$, $je_1=-e_0$.  Then $j$ satisfies the integrability condition $N_j\equiv 0$, that is, $j$ 
is a complex structure on $\ggo_C$. 
\end{remark}

\begin{proposition}\label{solv-complex2n}
Let $J\in M(2m, \RR)$ such that $J^2=-I_{2m}$ and let $B\in M(2m, \RR)$. Then $B$ is conjugate in $GL(2m, \RR)$ to a matrix in $M_J (2m, \RR)$ if and only if for any real eigenvalue $\lambda$ of $B$ the following condition is satisfied 
\begin{equation}\label{ast}
 \text{if }\, \Lambda (B(\lambda))=\begin{cases}
( n_1,\ldots , n_k; p_1, \ldots , p_k; t ), & \text{ then  } t \text{  and } p_i \text{  are even for all } i,\\
(m_{\lambda} ), & \text{ then }\, m_{\lambda} \text{  is even}.  
\end{cases} 
 \end{equation}
 In particular, if $B$ is conjugate to a matrix in $M_J (2m, \RR)$, then its semisimple and nilpotent parts are conjugate to matrices in $M_{J'} (2m, \RR)$ and $M_{J''} (2m, \RR)$, respectively, for some $J'$ and $J''$. 
\end{proposition}
\begin{proof}
Without loss of generality, we may suppose that $J=J_1 ^{\oplus m}$, where $J_1= \left(\begin{array}{cc}
      0  & -1  \\
      1 &  0 
      \end{array}\right)$. Let $\psi$ be the injective $\RR$-algebra homomorphism given by
$$
\psi : M(m, \CC) \rightarrow M(2m, \RR), \hspace{0.5cm} \psi (A)=\left(\mathcal{C}(b_{kj}, d_{kj})\right),  \text{ for } A=(b_{kj}+id_{kj}), 
$$  
that is, $\psi(A)$ is the real $2m\times 2m$ matrix consisting  of $m^2$ blocks of size $2\times 2$ of the form $\mathcal{C}(b_{kj}, d_{kj})$ as in \eqref{eq:complex-matrix}.
It is straightforward that Im $\psi = M_J(2m, \RR)$. Also, it is easy to check  that $\psi$ satisfies the following properties: 
\begin{enumerate}
    \item[(i)] $\psi (A)$ is conjugate to $\psi (C)$ for conjugate matrices $A$, $C \in M(m, \CC)$;
    \item[(ii)] $\psi (A\oplus C)=\psi (A)\oplus \psi (C)$ for square matrices $A$ and $C$;
    \item[(iii)] for real numbers $a,b$ with $b>0$,  $\psi (\mathcal{J}_n(a+ ib))=\mathcal{C}_n(a,b)$  (see \eqref{eq:matriz-inflada}) and  $\psi (\mathcal{J}_n(a))$ is conjugate to $\mathcal{J}_n(a)\oplus \mathcal{J}_n(a)$.
\end{enumerate}

 Now, for $B\in M_J (2m,\RR)$, the matrix $\psi^{-1}(B)\in M(m, \CC)$ is conjugate to its Jordan normal form $\mathcal J (\psi^{-1} (B))$. Hence, by property (i) above, $B$ is conjugate to $\psi(\mathcal J (\psi^{-1} (B)) )$. By conditions (ii) and (iii) we obtain the real Jordan normal form of $B$ (up to reordering). Finally, we have that the sets of distinct real eigenvalues of $\psi^{-1}(B)$ and $B$ coincide and if $\Lambda_{\CC}(\psi^{-1}(B)(\lambda))= ( m_1,\ldots , m_k; q_1, \ldots , q_k; s )$, where $\Lambda_{\CC}$ is defined  as in \eqref{Lambda} for a matrix in $M_{\lambda}( r, \CC)$, property (iii) implies that $\Lambda(B(\lambda))= ( m_1,\ldots , m_k; 2q_1, \ldots , 2q_k; 2s )$ for any real eigenvalue $\lambda$. Therefore, condition \eqref{ast} holds.

 For the converse, suppose that $B\in M(2m, \RR)$ satisfies condition  \eqref{ast} and let $$\text {spec} (B)= \{\lambda_1 , \ldots , \lambda _{\ell}, \lambda _{\ell +1}, \overline{\lambda} _{\ell +1}, \ldots, \lambda _s, \overline{\lambda }_s\},$$ where $\lambda_j\in \RR$, for $j=1,\ldots,\ell$, and $\lambda_j \in \CC$  with $\text{Im}(\lambda_j)>0$, for $j=\ell +1,\ldots,s$. The matrix $B$ is conjugate to $B(\lambda_1)\oplus \ldots \oplus B(\lambda_s)$ (see \eqref{Cdecomposition}). By hypothesis, $B(\lambda_j)\in M(m_{\lambda_j}, \RR)$ with $m_{\lambda_j}=2r_j$, for any $j=1, \ldots, \ell$.  Now, for each $j=\ell + 1, \ldots, s$, let $n_j\in \NN$ such that $B(\lambda_j)\in M( 2n_j, \RR)$.  By reordering the Jordan blocks in each $B(\lambda_j)$, for all $j$ such that $\lambda_j\in \RR$, we obtain that $B(\lambda_1)\oplus \ldots \oplus B(\lambda_s)$ is conjugate to a matrix in $M_J (2m, \RR)$ where $J=J_{r_1} \oplus \ldots \oplus J_{r_{\ell}}\oplus  J_1^{\oplus n}$ (see \eqref{eq:Jm} for the definition of the complex structures) with $n=n_{\ell +1} + \ldots + n_{s}$, and the converse follows.

 Finally, decompose $B$  as $B=B_s+B_n$ with $B_s$ semisimple and $B_n$ nilpotent. It is easily verified that if $B$ satisfies condition \eqref{ast}, then $B_n$ and $B_s$ also satisfy this condition, and the last assertion follows.
\end{proof}

\subsection{The  nilpotent case}\label{sec:nilp-complex}

In the next proposition we exhibit a relationship between the dimension of a nilpotent almost abelian Lie algebra $\ggo_A$ and its step of nilpotency when $\ggo_A$ admits a complex structure.

\begin{proposition} \label{prop:nilp-no-complex}
Let $\ggo_A=\RR e_0 \ltimes_A \RR^{2n-1} $ be a $k$-step nilpotent  almost abelian  Lie algebra. If $\ggo_A$  admits  a complex structure, then   $k\leq n$.       
\end{proposition}
\begin{proof}  If $\ggo_A$ admits a complex structure, then $A$ is conjugate to a matrix $C$ of the form \eqref{eq:complex-abelian} with $a=0$ and   $ J=J_{n-1}$ as in \eqref{eq:Jm} (see Remark \ref{rem:nilp-similiar} and Remark \ref{rem:complex}). In particular, $B\in  M_{J_{n-1}}(2(n-1), \RR)=\Im\,\varphi$, 
  with  $\varphi: M(n-1,\CC)\to M_{J_{n-1}}(2(n-1), \RR)$ as in \eqref{eq:alg-iso}.  Since $\ggo_A$ is $k$-step nilpotent, $C^k=0$ and $C^{k-1}\neq 0$, which implies that $B^k=0$.  
 
 If $B^{k-1} \neq 0$, then  
 $\varphi ^{-1}(B)\in M (n-1, \CC )$ satisfies $(\varphi ^{-1}(B))^k=0$ and $(\varphi ^{-1}(B))^{k-1}\neq 0$, that is, the minimal polynomial of $\varphi ^{-1}(B)$ is $x^k$. Therefore, $k\leq n-1$.

 If $B^{k-1} = 0$, then $B^{k-2} \neq 0$ since $ C^{k-1}\neq 0$. In this case we have that 
 $(\varphi ^{-1}(B))^{k-1}= 0$ and  $(\varphi ^{-1}(B))^{k-2}\neq 0$. Then, the minimal polynomial of $\varphi ^{-1}(B)$ is $x^{k-1}$,  hence $k-1\leq n-1$, and the corollary follows.   
\end{proof}
When $\ggo_A=\RR e_0  \ltimes_A \RR^{2n-1} $ is $2$-step nilpotent almost abelian, then $\ggo_A$  always admits a complex structure, as the next result shows. 
\begin{proposition}\label{prop:2-step-complex} Let $\ggo_A=\RR e_0 \ltimes_A \RR^{2n-1}$ be a $2$-step nilpotent  almost abelian Lie algebra. Then $\ggo_A$ admits a complex structure. 
\end{proposition}
\begin{proof} Since $\ggo_A$ is $2$-step nilpotent, the Jordan normal form of  $A$ is  $\J_2  ^{ \oplus p}\oplus 0_s $, for $p,\, s>0$.
It follows that $s=2k+1$ with $k= n-p-1$. 

If $p$ is even, $p=2l$, then 
by a change of basis, we obtain that $A$ is conjugate to $C= 0_1\oplus \J_2 ^{\oplus p}\oplus 0_{2k}$, and clearly this has the form \eqref{eq:complex-abelian} with $a=0$, $v=0$ and $B=\J_2 ^{\oplus p}\oplus 0_{2k} \in M_{ J}(2(n-1), \RR)$, for $J= J_{2l} \oplus \tilde{J}$, where $\tilde{J}$ is an arbitrary  $2k \times 2k$ matrix such that $\tilde{J}^2=-I_{2k}$.

If $p$ is odd, $p=2l+1$, then 
by a change of basis, we obtain that $A$ is conjugate to $C= \J_2 \oplus 0_{2k+1}\oplus \J_2 ^{\oplus (p-1)}$ which has the form \eqref{eq:complex-abelian}.  In fact:
\[ 
C = \begin{pmatrix}
    0  & 0 \\
      \epsilon_1   & B
\end{pmatrix},
\] with $B= 0_{2k+2}\oplus \J_2 ^{\oplus (p-1)} \in M_{ J}(2(n-1), \RR)$, for $J= \tilde{J} \oplus J_{2l}$, 
where $\tilde{J}$ is an arbitrary  $2(k+1) \times 2(k+1)$ matrix such that $\tilde{J}^2=-I_{2(k+1)}$.
\end{proof}
The above proposition does not hold for  general  $3$-step nilpotent almost abelian Lie algebras, as the next example shows. 
\begin{example} Let $\ggo_M=\RR e_0\ltimes_M \RR^5$ with $M= \J_3\oplus 0_2$. This Lie algebra was denoted by $(0,0,0,0,12,15)$ in \cite{Sal}, where it was shown that it  admits no complex structure.    
\end{example}

However, under an additional assumption, we are able to prove the next result.

\begin{proposition}\label{prop:complex-3step}  Let $\ggo_A=\RR e_0 \ltimes_A \RR^{2n-1}$ be a $3$-step nilpotent  almost abelian Lie algebra and set $p_1:=(2n-1)-d_2, \; p_2:= 2d_2-d_1-(2n-1),  $ where $d_\ell=\dim( \ker A^\ell), \, \ell=1,2$. Assume that 
 any of the following conditions is satisfied: 
 \begin{enumerate}
 \item[$\ri$] $p_2$ is odd, 
 \item[$\rii$] $p_1$ and $p_2 $ are both  even. 
 \end{enumerate}
 Then $\ggo_A$ admits a complex structure.     
\end{proposition}
\begin{proof}
 Since $\ggo_A$ is $3$-step nilpotent, the Jordan normal form of $A$ is $\J_3^{\oplus q_1}\oplus \J_2^{\oplus q_2}\oplus 0_s$ with $q_1>0, \, q_2,\,  s\geq 0$. We show next that $q_\ell=p_\ell, \, \ell =1,2$. In fact, it follows that 
 \[ 2n-1=3q_1+2q_2+s, \qquad d_1=q_1+q_2+s, \qquad d_2=2(q_1+q_2)+s, 
 \]
 from which we get
 \[ (2n-1)-d_1=2q_1+q_2, \qquad (2n-1)-d_2=q_1.
 \]
 Therefore,
 \[
 q_1= p_1, \qquad q_2=(2n-1)-d_1-2q_1=2d_2-d_1-(2n-1)=p_2 , 
 \]
 as  asserted. 
 
 Assume that $\ri$ holds, that is,  $p_2=2r +1$ is odd. We consider separately the cases $p_1$ even or $p_1$ odd.

 If $p_1=2k+1$ is odd, then $s$ is even and $A$ is conjugate to $C=\J_3\oplus\J_2 \oplus \J_2^{\oplus (2r)}\oplus\J_3^{\oplus (2k)}\oplus 0_s $, which has the form \eqref{eq:complex-abelian}. In fact:
 \[ C = \begin{pmatrix}
    0  & 0 \\
      \epsilon_1   & B
\end{pmatrix},
\] with $B=  \J_2 ^{\oplus (2r+2)}\oplus \J_3^{\oplus (2k)}\oplus 0_s \in M_{ J}(2(n-1), \RR)$, for $J= J_2^{\oplus (r+1)}\oplus J_3^{\oplus k}\oplus \tilde J$, where $\tilde J$ is an arbitrary $s\times s$ matrix such that $(\tilde J)^2=-I_{s}$.

 If $p_1=2k$ is even, then $s$ is odd and $A$ is conjugate to $C=\J_2\oplus 0_{s}\oplus \J_3^{\oplus (2k)}\oplus \J_2^{\oplus (2r)} $, which has the form \eqref{eq:complex-abelian}. In fact:
 \[ C = \begin{pmatrix}
    0  & 0 \\
      \epsilon_1   & B
\end{pmatrix},
\] with $B= 0_{s+1}\oplus  \J_3^{\oplus (2k)}\oplus \J_2^{\oplus (2r)}   \in M_{ J}(2(n-1), \RR)$, for $J=\tilde J\oplus J_3^{\oplus k} \oplus J_2^{\oplus r}$  
where $\tilde J$ is an arbitrary $(s+1)\times (s+1)$ matrix such that $(\tilde J)^2=-I_{(s+1)}$.

 If $\rii$ is satisfied, that is, $p_1$ and $p_2$ are both even, then $s\geq 1$ is odd.  Let $p_1=2k$, $p_2=2r$, then $A$ is conjugate to 
 \[C= \begin{pmatrix}
    0  & 0 \\
      0   & B
\end{pmatrix},\]
 with $B= \J_3^{\oplus (2k)}\oplus \J_2^{\oplus (2r)} \oplus 0_{s-1}  \in M_{ J}(2(n-1), \RR)$, for $J= J_3^{\oplus k} \oplus J_2^{\oplus r} \oplus\tilde J$  
where $\tilde J$ is an arbitrary $(s-1)\times (s-1)$ matrix such that $(\tilde J)^2=-I_{(s-1)}$.
\end{proof}
As a consequence of Theorem \ref{classifcomplex} below it will turn out that  the converse of the above proposition holds.

Given a nilpotent matrix $M$, in the following theorem we give necessary and sufficient conditions on the tuple $\Lambda (M)$ to obtain pairwise non-isomorphic nilpotent almost abelian Lie algebras $\ggo_M$ admitting complex structures (compare with~\cite[~Appendix A]{AB}). We exclude the zero matrix $M$, since when $M=0$, $\ggo_M$ is the abelian Lie algebra, which always admits a complex structure. 

\begin{theorem}\label{classifcomplex}
Let $\ggo_M=\RR e_0  \ltimes_M \RR^{2n-1}$ be the   $2n$-dimensional  nilpotent almost abelian Lie algebra with $\Lambda(M)=(n_1,\ldots, n_k; p_1, \ldots, p_k;t)$. Then  $\ggo_M$ admits a complex structure if and only if one of the following conditions holds:
\begin{enumerate} 
\item[$\ri$] $t$ is odd and $p_i$ is even for all $i=1,\dots,k$,
\item[$\rii$] $t$ is odd, $p_i$ is even for all $i=1, \dots,k-1$, $n_k=2$ and $p_k$ is odd,
\item[$\riii$] $t$ is even, $k\geq 2$, there exists $2\leq \ell \leq k$ such that $n_{\ell-1}=n_{\ell} +1$, $p_{\ell-1}$ and $p_{\ell}$ are odd, and $p_i$ is even for all $i\notin \{\ell-1,\ell\}$.
\end{enumerate}
\end{theorem}
\begin{proof}
Assume that $\mathfrak{g}_M $ admits a complex structure, then by Remark \ref{rem:complex} and Remark \ref{rem_complex-conj}, $M$
is conjugate to a matrix 
$$
 C= \begin{pmatrix} 
0 & 0 \\
v & B
\end{pmatrix},  \text{ for some } v \in \RR^{2(n-1)} \text{ and } B \in M_{J_{n-1}}(2(n-1) , \RR). 
 $$
Then, either $B=0$ and $v\neq 0$, or $B\neq 0$. If $B=0$, Corollary \ref{posibles_C} implies that $C$, hence $M$, is conjugate to $\J_2\oplus 0_{2n-3}$. Therefore, $M$ satisfies condition (ii).

If $B\neq 0$, by Proposition \ref{solv-complex2n} $B$ is conjugate to 
$B ^\prime =  \J_{m_1}^{\oplus 2q_1}  \oplus \cdots \oplus \J_{m_r}^{\oplus 2q_r} \oplus  0_{2s}$, where $m_1> \dots > m_r\geq 2$. 
 Hence, $M$ is conjugate to 
 $$
 C^\prime=  \begin{pmatrix} 
0 & 0 \\
v^\prime & B^\prime
\end{pmatrix},  \text{ for some } v^\prime  \in \RR^{2(n-1)}. 
 $$
 Now, according to Corollary \ref{posibles_C}, we have the following three possibilities: 
 
 (a) $M$ is conjugate to $0_1 \oplus B^\prime$. That is, 
 $$
 \J_{n_1}^{\oplus p_1}  \oplus \cdots \oplus \J_{n_k}^{\oplus p_k} \oplus  0_{t}\; \text{ is conjugate to }\; 
 \J_{m_1}^{\oplus 2q_1}  \oplus \cdots \oplus \J_{m_r}^{\oplus 2q_r} \oplus  0_{2s+1}.
$$
 Hence, by the uniqueness, up to reordering, of the Jordan normal form, it follows that  
 $k=r$, $n_i= m_i$ and $p_i= 2q_i$, for all $i \in \{ 1, \dots, k\}$ and $t= 2s+1$.  Thereby, condition (i) of the theorem is satisfied. 
 
(b) $M$ is conjugate to 
$$
\J_{m_1}^{\oplus 2q_1}  \oplus \cdots \oplus 
\J_{m_{d-1}}^{\oplus 2q_{d-1} }\oplus  \J_{m_d+1} \oplus 
\J_{m_{d}}^{\oplus 2q_d-1}
\oplus \dots \oplus 
\J_{m_r}^{\oplus 2 q_r} \oplus  0_{2s},
$$
 for some $d \in \{ 1, \dots, k \}$. Notice that $m_{d-1} \geq m_d+1 > m_d$. For both cases, $m_{d-1}= m_d+1$ or $m_{d-1}> m_d+1$, condition $\riii$ of the theorem is satisfied. 
 
 (c)  $M$ is conjugate to  
  $$\J_{m_1}^{\oplus 2q_1}  \oplus \cdots \oplus \J_{m_r}^{\oplus 2q_r} 
  \oplus \J_2 
  \oplus  0_{2s-1}.
 $$
 That is, 
 $$
  \J_{n_1}^{\oplus p_1}  \oplus \cdots \oplus \J_{n_k}^{\oplus p_k} \oplus  0_{t}\; \text{ is conjugate to }\;
 \J_{m_1}^{\oplus 2q_1}  \oplus \cdots \oplus \J_{2}^{\oplus 2q_{r^\prime} +1} 
  \oplus  0_{2s-1}.
 $$
 where $r^\prime = r $ if $m_r=2$, and $q_{r^\prime}=0$ otherwise. Finally, by the uniqueness, up to reordering, of the Jordan normal form,  condition $\rii$ in the theorem is satisfied. 

For the converse, by Remark \ref{rem:complex} it is enough to show (case by case)  that $M$ is conjugate to a matrix $C$ as in ~\eqref{eq:complex-abelian}. 

$\ri$ Assume that $t=2s+1$ and $p_i= 2 q_i$, for $ i = 1, \dots, k $. 
Then, $M$ is conjugate to
$C = 0_1 \oplus B$ where $
B=  \J _{n_1}^{\oplus p_1}  \oplus \cdots \oplus \J _{n_k}^{\oplus p_k} \oplus  0_{2s} $.  Since  $B$   satisfies  the conditions of Proposition \ref{solv-complex2n}, $B$ is conjugate to a matrix in $M_{J_{n-1}}(2(n-1), \RR)$.
 
 $\rii$  Analogous to the previous case, assume that 
 $t=2s+1$,  $p_i= 2 q_i$, for $i = 1, \dots, k-1$, and $n_k=2$ with $p_k= 2q_k+1$. Then, $M$ is conjugate to 
 $$
 C=\J_2  \oplus 0_{2s+1}  \oplus \J_{n_1}^{\oplus 2q_1}  \oplus \cdots \oplus \J_{2}^{\oplus 2q_k}
 $$ 
 and it is of the form
 $$
    C= \begin{pmatrix}
        0 & 0\\
        \epsilon_1 & B
    \end{pmatrix}, 
   \quad \text{ where } \quad 
   B= 0_{2(s+1)}\oplus \J _{n_1}^{\oplus 2q_1}  \oplus \cdots \oplus \J _{2}^{\oplus 2q_k} \; \text{or} \; B=0_{2(n-1)}.
 $$
Again, the matrix $B$ satisfies  the conditions of Proposition \ref{solv-complex2n}, so it  is conjugate to a matrix in
$M_{J_{n-1}}(2(n-1), \RR)$.
 
$\riii$ Assume that $t=2s$ and let $\ell $  be as in the hypothesis. Then, 
$p_i= 2 q_i$, for $ i \notin \{\ell -1, \ell \}$,  
$p_{\ell-1 } = 2 q_{\ell -1}+1$ and
 $p_{\ell} = 2 q_{\ell}+1$.  Then $M$ is conjugate to the matrix
 
  $$
    C= \begin{pmatrix}
        0 & 0\\
        \epsilon_1 & B
    \end{pmatrix},
  $$
  where 
 $$B=
  \J _{n_{\ell}}^{\oplus 2(q_{\ell}+1)}
  \oplus \J _{n_1}^{\oplus 2q_1} 
     \oplus \cdots \oplus   
     \J _{n_{\ell -1}}
     ^{\oplus 2 q_{\ell -1} } \oplus  
     \J _{n_{\ell +1}} ^{\oplus 2q_{\ell +1} }
     \oplus  \dots  \oplus 
    \J _{n_k}^{\oplus 2q_k}  \oplus   0_{2s}.
 $$
Again, the matrix $B$ also satisfies  the conditions of Proposition \ref{solv-complex2n}, so it  is conjugate to a matrix in $M_{J_{n-1}}(2(n-1), \RR)$.
\end{proof}

For each $n\in \NN$, let 
\begin{equation*}
\begin{split}  
    {\bigwedge}^c_{2n-1}:&= \{ (n_1,\ldots , n_k; p_1, \ldots , p_k; t ) \in {\bigwedge}_{2n-1}:  \text{satisfying (i), (ii)} \\ & \qquad \text{or (iii) of Theorem \ref{classifcomplex}}\} \cup\{(2n-1)\}.
\end{split} 
\end{equation*}
Consider the bijection  $\tilde{\Lambda}$ 
defined in \eqref{eq:bijection-tuples} for $b=0$ and $m=2n-1$. As a consequence of Lemma \ref{lem:isom-nilpot} and Theorem \ref{classifcomplex},  we obtain the following classification result.

\begin{corollary}\label{cor:classif-complex} 
Let $n\in \NN$ and let $\ggo_M=\RR e_0  \ltimes_M \RR^{2n-1}$ be a    $2n$-dimensional  nilpotent almost abelian Lie algebra. Then $\ggo_M$ admits a complex structure if and only if $\Lambda(M)\in {\bigwedge}^c_{2n-1}$. Moreover, 
the isomorphism classes of $2n$-dimensional   nilpotent almost abelian Lie algebras admitting a complex structure are parametrized by the set $\bigwedge^c_{2n-1}$. \end{corollary}

\begin{example}\label{ex:complex-8}
    We obtain that  the set  ${\bigwedge}^c_{7}$ consists of the following tuples:
    \[
     (4,3;1,1;0), \, (3;2;1), \, (3,2;1,1;2),\, (2;3;1), \, (2;2;3), \, (2;1;5),\, (7) .
    \]
    Therefore, there are 7 isomorphism classes of 8-dimensional nilpotent almost abelian Lie algebras admitting complex structures. The set $ {\bigwedge}_{7\,} - {\bigwedge}^c_{7}$  contains 8 tuples, that is, there are 8 isomorphism classes of 8-dimensional nilpotent almost abelian Lie algebras not admitting complex structures.
    For instance, since $(7;1;0)\notin {\bigwedge}^c_{7}$, the corresponding nilpotent  almost abelian Lie algebra does not admit a complex structure.
\end{example}

\begin{example}\label{ex:complex-10} The set  ${\bigwedge}^c_{9}$ is given by \[
\begin{array}{llllll}
(5,4;1,1,0),&   (4;2;1),&   (4,3;1,1;2),&   (3,2: 2, 1; 1),& (3; 2; 3),& (3, 2; 1,3; 0), \\
(3, 2 ; 1,1; 4),& (2; 4;1),& (2; 3;3),& (2;2;5),& (2;, 1, 7),& (9).
\end{array}\]
This gives that there are 12 isomorphism classes of 10-dimensional nilpotent almost abelian Lie algebras admitting complex structures. 
It can be shown that the set $ {\bigwedge}_{9\,} - {\bigwedge}^c_{9}$ contains 18 tuples, that is, there are 18 isomorphism classes of 10-dimensional nilpotent almost abelian Lie algebras not admitting complex structures. For instance, 
 $(9;1;0) \notin {\bigwedge}^c_{9}$.
\end{example}

\begin{example} When $M=\J_{2n-1}$, with $n\geq 2$, the Lie algebra $\ggo_M$ is a filiform Lie algebra. Since $(2n-1;1;0)\notin  {\bigwedge}^c_{2n-1}$,  $\ggo_M$ does not admit a complex structure, a fact that was already known  (see \cite{GR}). 
\end{example} 

\subsection{The general case}
We start by proving the following result,  which will be needed in Theorem \ref{complex_M(a)} and Theorem \ref{complex_M(a)+Q}. 
\begin{proposition}\label{odd-space}
Let $M\in M(2n-1, \RR )$ and assume that $\ggo_M$ admits a complex structure, then there exists a unique $a\in \text{spec}_\RR (M)$ such that $m_a$ is odd. In particular, 
      $M$ is conjugate to either $M(a)$  or 
  $M(a)\oplus Q$, where $Q\in M(2n-1-m_a, \RR )  $ and $a\not\in \emph {spec}\, (Q)$, depending on whether  $m_a=2n-1$ or $m_a<2n-1$.
\end{proposition}
\begin{proof}
By Remark \ref{rem:complex} $M$ is conjugate to $C= \begin{pmatrix}
      a   & 0 \\
      v   & B
     \end{pmatrix}$, where $a\in \RR$, $v\in \RR^{2n-2}$ and $B\in M_J (2n-2, \RR)$, for some $J$. Since $M$ and $C$ are conjugate and $a$ is an eigenvalue of $C$, then $a$ is an eigenvalue of $M$. By Proposition \ref{solv-complex2n}, $a$ is the unique real eigenvalue of $M$ such that $m_a$ is odd. The last sentence of the statement follows directly from Lemma \ref{M(a)+Q}. 
\end{proof}

In view of Proposition \ref{odd-space} we state the following theorems. In the first one, which is a consequence of Corollary \ref{cor:classif-complex}, we consider the case when $M$ has a unique eigenvalue $a$, with $a\in \RR$.

\begin{theorem} \label{complex_M(a)} 
    Let $M\in M_a(2n-1, \RR)$, with $a\in \RR$. Then, $\ggo_M$ admits a complex structure if and only if $\Lambda (M)\in {\bigwedge}^c_{2n-1}$.
\end{theorem}
\begin{proof}

Suppose that $\ggo_M$ admits a complex structure, then by Remark \ref{rem:complex}, $M$ is conjugate to $C= \begin{pmatrix}
      a   & 0 \\
      v   & B
     \end{pmatrix}$, where $v\in \RR^{2n-2}$ and $B\in M_J (2n-2, \RR)$ for some $J$. In particular, $B\in M_a(2n-2, \RR)$. We consider the nilpotent matrix $N= C-aI_{2n-1}= \begin{pmatrix}
      0   & 0 \\
      v   & N'
     \end{pmatrix}$, where $N' = B-aI_{2n-2}$. Since $N' \in M_J (2n-2, \RR)$, by Remark \ref{rem:complex} we obtain that $\ggo_N$ admits a complex structure. Then, Corollary \ref{cor:classif-complex} implies that $\Lambda (N)\in {\bigwedge}^c_{2n-1}$. Finally, using Remark \ref{L(A)=L(N)} and the fact that $\Lambda (C)=\Lambda (M)$, we obtain that $\Lambda (M)\in {\bigwedge}^c_{2n-1}$.

Conversely, suppose that $\Lambda (M) \in {\bigwedge}^c_{2n-1}$ and consider as above, the nilpotent matrix $N=M-aI_{2n -1}$, then by Remark \ref{L(A)=L(N)} we have $\Lambda(N)\in {\bigwedge}^c_{2n-1}$. Therefore, Corollary \ref{cor:classif-complex} implies that $\ggo_N$ admits a complex structure. In particular, $N$ is conjugate to a matrix $\left(\begin{array}{cc}
      0  & 0  \\
      v &  N' 
      \end{array}\right)$, where $N' \in M_{J} (2n -2, \RR)$   for some $J$. So, $M$ is conjugate to $\left(\begin{array}{cc}
      a  & 0  \\
      v &  B 
      \end{array}\right)$, where $B= N' + a I_{2n -2} \in M_{J} (2n -2, \RR)$. In consequence, by Remark \ref{rem:complex},  $\ggo_M$ admits a complex structure.
      \end{proof}

In the next theorem we consider the case when $M$ has two or more distinct eigenvalues and at least one of them is real.
      
\begin{theorem}\label{complex_M(a)+Q}    
 Let $M\in M(2n-1,\RR)$ such that $M=M(a) \oplus Q$ where $a\not\in \emph{spec}\, (Q)$, $M(a)\in M(m_a,\RR)$, with $a\in \RR$ and $m_a$ odd. Then, $\ggo_M$ admits a complex structure if and only if $\Lambda (M(a))\in {\bigwedge}^c_{m_a}$ and $Q\in M_{J'}(m,\RR)$, for some $J'$ such that $(J')^2=-I_m$ with $m=2n-1-m_a$.
\end{theorem}
\begin{proof} 
First, assume that $M(a)=aI_{m_a}$, then $M=aI_{m_a} \oplus Q$ with $m_a$ odd. If $m_a=1$, then $\ggo_M$ admits a complex structure if and only if $Q\in M_{J'}(m,\RR)$, for some $J'$. On the other hand, if $m_a>1$ we have that $m_a-1>0$ is even and $a I_{m_a -1} \oplus Q\in M_{J'}(2n-2, \RR)$, for some $J'$, and this holds if and only if $Q\in M_{J'}(m,\RR)$, for some $J'$. Therefore, by Remark \ref{rem:complex},  $\ggo_M$ admits a complex structure if and only if $Q\in M_{J'}(m,\RR)$, for some $J'$.  

Now, assume that $M(a)\neq aI_{m_a}$. If $\ggo_M$ admits a complex structure, then by Remark \ref{rem:complex}, $M$ is conjugate to $C= \begin{pmatrix}
      a   & 0 \\
      v   & B
     \end{pmatrix}$, where $v\in \RR^{2n-2}$ and $B\in M_J (2n-2, \RR)$ for some $J$. Therefore, by Proposition \ref{solv-complex2n},  any real eigenvalue of $B$ satisfies  condition \eqref{ast}. Since $M(a) \neq aI_{m_a}$, we have that $m_a >1$. Hence,  $a\in \text{spec}_\RR (B)$ and $B$ is conjugate to $B(a)\oplus  Q'$, with $a\notin \text{spec}\,(Q') $. Then, any real eigenvalue of both, $B(a)$ and $Q'$, satisfies  condition \eqref{ast}. In particular, by Proposition \ref{solv-complex2n}, $B(a) \in M_{J_1}(m_a-1,\RR)$ and $Q' \in M_{J_2}(m,\RR)$, for some $J_1$ and $J_2$. 
   Since $a\notin \text{spec}(Q') $, Lemma~\ref{B(a)+D} implies that $C$ is conjugate to 
   $
   \left(
   \begin{array}{c|c|c}
      a  & 0  & 0\\
      \hline
      u &  B(a) & 0 \\
      \hline 
      0  & 0 & Q'
      \end{array}
     \right)
  $, where $u\in \RR ^{m_a -1}$, $ B(a) \in M(m_a -1,\RR)$,  $Q' \in M(m,\RR)$. 

Let $C_a:=\left(\begin{array}{cc}
      a  & 0  \\
      u &  B(a) 
      \end{array}\right)\in M(m_a, \RR)$.
Since $B(a) \in M_{J_1}(m_a -1,\RR)$, by Remark \ref{rem:complex}, $\ggo _{C_a}$ admits a complex structure. Then applying Theorem \ref{complex_M(a)} we have $\Lambda (C_a)\in {\bigwedge}^c_{m_a}$. Finally, since $M = M(a) \oplus Q$ is conjugate to $C_a\oplus Q'$, we obtain that $\Lambda (C_a)= \Lambda (M(a))$ and $Q'$ is conjugate to $Q$. Hence, the statement follows.

Conversely, suppose that $\Lambda (M(a)) \in {\bigwedge}^c_{m_a}$, then by Theorem \ref{complex_M(a)}, we obtain that $\ggo_{M(a)}$ admits a complex structure. In particular, $M(a)$ is conjugate to a matrix $\left(\begin{array}{cc}
      a  & 0  \\
      v &  B 
      \end{array}\right)$, where $B\in M_{J_1} (m_a -1, \RR)$ for some $J_1$. 
      On the other hand, $Q\in M_{J'}(m, \RR)$ for some $J'$, since any real eigenvalue of $Q$ satisfies condition \eqref{ast}.
      %of Proposition \ref{solv-complex2n}. 
      Finally, by Remark \ref{rem:complex}, $\ggo_M$ admits a complex structure since $M$ is conjugate to $\left(\begin{array}{c|c|c}
      a  & 0  & 0\\
      \hline
      v &   B & 0 \\
      \hline
      0 & 0 & Q
      \end{array}\right)$, where  $B \oplus Q \in M_J (2n-2, \RR)$ with $J=J_1\oplus J'$.  \end{proof}

\

Let $A\in M(2n-1, \RR)$ be a non-nilpotent matrix. By combining   Theorem \ref{complex_M(a)}, Theorem \ref{complex_M(a)+Q} and  Lemma \ref{lem:isom-classes} it follows that  
$\ggo_A$ admits a complex structure if and only if $A$ is conjugate to $b M$ for some  $b\in\RR, \, b\neq 0$, where either $1\in\text{spec}\, (M)$ or $0\in\text{spec}\, (M)$ and $M$ takes one of the  forms below:
\begin{enumerate}
 \item $M=M(1)$  with $\Lambda (M(1))\in \bigwedge ^c_{2n-1}$,
    \item  $M=M(1) \oplus Q$ where   $\Lambda (M(1))\in \bigwedge ^c_{m_1}$ with $m_1$  odd and $Q\in M_{J_r} (2r, \RR)$ with $2r=2n-1-m_1 $ and $1\notin \text {spec}\, (Q)$,
    \item $M=M(0)\oplus Q$ where $\Lambda (M(0))\in \bigwedge ^c_{m_0}$ with $m_0$ odd  and $Q\in M_{J_r} (2r, \RR)$ with $2r=2n-1-m_0$, $0\notin \text {spec} \,(Q)=\{ \mu_1, \ldots , \mu_s \}$ and $1=|\mu _1| \geq \cdots  \geq |\mu_s |>0$. 
    \end{enumerate}\ 

We point out that, since $A$ is conjugate to $bM$ with $b\neq 0$,  it follows from Proposition \ref{ad-conjugated} that $\ggo_A$ is isomorphic to $\ggo_M$. Therefore, 
the isomorphism classes of $\ggo_A$ are parametrized by the isomorphism classes of $\ggo_M$. In order to obtain this parametrization, we  define the following quotient spaces:
\[ \mathcal C^1_r:=\{ Q\in M_{J_r} (2r, \RR) \,:\, 1 \notin \text {spec}\, (Q) \} /GL(2r,\RR ),  \]
\[   \mathcal{C}^0_r:= \{Q\in M_{J_r} (2r, \RR) : 0\notin \text {spec}\, (Q)  =\{ \mu_1, \ldots , \mu_s \} \text{ and } 1=|\mu _1| \geq \cdots  \geq |\mu_s |>0   \}/ _\sim
\]
where $Q_1\sim Q_2$ if and only if $Q_1$ is  $GL(2r,\RR )$-conjugate to $\pm Q_2$ (see Lemma \ref{lem:isom-classes}).

\begin{corollary} Let $A\in M(2n-1, \RR)$ be a non-nilpotent matrix. The isomorphism classes of $\ggo_A$  admitting a complex structure are
the  isomorphism classes of $\ggo_M$ for $M$ as in ${\rm (1)}$, ${\rm (2)}$ or ${\rm (3)}$ above, which are parametrized, respectively, by:   
\begin{equation*}
   \text{${\rm (1)}$ } \;  {\bigwedge} ^c_{2n-1},
    \qquad \text{${\rm (2)}$ } \; \bigcup_{k=1}^{n-1} \, \left({\bigwedge} ^c_{{2k-1}}\times\; \mathcal C_{n-k}^1\right),  \qquad 
    \text{${\rm (3)}$ } \;\bigcup_{k=1}^{n-1} \, \left({\bigwedge} ^c_{{2k-1}}\times\; \mathcal C_{n-k}^0\right).    
\end{equation*}
\end{corollary}

    \

\section{ Classification of almost abelian Lie algebras  admitting symplectic structures }

\subsection{Characterization of almost abelian Lie algebras admitting a symplectic  structure}

We consider the Lie subalgebra of $M(2m, \RR)$ given by
\begin{equation}\label{eq:sympl}
 \begin{split}
    \spg (2m, \RR)  & :=  \left\{ D \in M(2m, \RR) : D^T J_m + J_m D=0\right \} \\
      & =  \left\{ \begin{pmatrix} 
X & Y \\
Z & -X^T
\end{pmatrix} : X,Y,Z \in  M(m, \mathbb{R}), \; Y, Z \text{ symmetric} \right\},
  \end{split}
\end{equation}
where $J_m$ is the matrix defined in \eqref{eq:Jm} and $R^T$ denotes the transpose of a matrix $R$.

\begin{remark}\label{rem-symp-conj}
In general, each non-singular  skew-symmetric $\Omega \in M(2m, \RR)$  determines a Lie subalgebra of $M(2m, \RR)$ given by
\begin{equation}\label{Omega_J}
\spg_{\Omega}(2m, \RR):= \left\{D\in M(2m, \RR): D^T \Omega + \Omega D=0\right \}, 
\end{equation}
and 
it turns out that each $\spg_{\Omega}(2m, \RR)$ is conjugate  to $\spg (2m, \RR)$ in $GL(2m, \RR)$. 
More precisely, if $\Omega= Q^T J_m Q$ for some  $Q\in GL(2m, \RR) $, then $D\in \spg_\Omega (2m, \RR)$ if and only if $QDQ^{-1}\in \spg (2m, \RR)$. Moreover, since $D^T\in \spg (2m, \RR)$ for all $D\in \spg (2m, \RR)$, it follows that  $B\in \spg_{\Omega_1} (2m, \RR)$ if and only if  $B^T\in \spg_{\Omega_2}(2m, \RR)$, where $\Omega_1= P^T J_m P$ and $\Omega_2= P^{-1} J_m (P^{-1})^T$, for some $P\in GL(2m, \RR)$. 
\end{remark}
Given a $2m$-dimensional real vector space $V$ and a non-degenerate 2-form $\omega$ on $V$, we define
\[\spg (V, \omega):=\{ S\in \glg(V) :\omega (Sx,y)+ \omega (x,Sy)=0, \text{ for all } x,y\in V \}.\]
Given a basis $\mathcal B$ of $V$ and $S\in \spg (V,\omega )$, we denote by $D$ and $\Omega$ the matrices of $S$ and $\omega$ with respect to $\mathcal B$, respectively. Then, using the notation introduced in \eqref{Omega_J}, $D \in \spg_{\Omega}\left(2m, \RR\right)$.

The following fact will be used frequently in the sequel.  

\begin{remark}\label{suma w_i} 
If $D_i \in \spg _{\Omega_i} (2m_i, \RR)$, for $i=1,2$, then $D=D_1 \oplus D_2\in \spg _\Omega (2m, \RR)$, where $\Omega=\Omega_1 \oplus \Omega_2$ and $m=m_1 +m_2$. Therefore, we have  an inclusion \[ \spg _{\Omega_1} (2m_1, \RR)\oplus \spg _{\Omega_2} (2m_2, \RR)\subset \spg _\Omega (2m, \RR).\] 
\end{remark}

The next proposition gives a characterization of the $2n$-dimensional almost abelian Lie algebras admitting a symplectic structure (see also \cite[Proposition 4.1]{LW}).
\begin{proposition}\label{prop:symp}
Let $\ggo$ be a $2n$-dimensional almost abelian Lie algebra with  codimension-one abelian ideal $\hg$. If $\omega$ is a symplectic structure on $\ggo$, then there exist $e_0\notin \hg, \, e_1\in \hg, \, c\in \RR $ and a subspace $\mathfrak u$ of $\hg$  such that  $\hg= \RR e_1\oplus \mathfrak{u}$, $\omega_\mathfrak{u}:= \omega|_{\mathfrak{u}\times \mathfrak{u}}$ is non-degenerate, 
$\omega = e^0 \wedge e^1 + \omega_\mathfrak{u}$ and $[e_0,e_1]=ce_1$. Moreover, $\Pi \circ \left(\text{ad}_{e_0}|_\mathfrak u \right)\in \spg (\mathfrak u, \omega_\mathfrak{u})$, where $\Pi :\hg \to \mathfrak u$ is the projection on $\mathfrak u$ along $\RR e_1$. 
 \end{proposition}

\begin{proof}
Consider the orthogonal complement $\hg^{\perp_{\omega}}$ of $\hg$  with respect to $\omega$. Since $\dim\hg = 2n-1$, we have that $\dim \hg^{\perp_{\omega}} =1$ and $\hg^{\perp_{\omega}} \subset \hg.$ Let $e_1 \in \hg$ such that $\hg^{\perp_{\omega}} = \vspan\{e_1\}$. Since $\omega$ is non-degenerate, there exists $e_0 \notin \hg$ such that $\omega(e_0,e_1)=1$. Setting $V=\vspan\{e_0,e_1\}$, we have $\vspan\{e_1\} \subset V$, which implies $\mathfrak u :=V^{\perp_{\omega}} \subset \hg$. Then, $\omega_\mathfrak u$ is non-degenerate and $\omega = e^0 \wedge e^1 + \omega_\mathfrak u$. 

We show next that if $[e_0,e_1]=ce_1 +v$ with $v\in \mathfrak u $, then $v=0$. Since $\omega$ is closed and $\hg$ is abelian, for $z\in \mathfrak u$ we have
\[ 
0=d\omega(e_0,e_1,z)=-\omega([e_0,e_1],z)- \omega([z,e_0],e_1)=- \omega(v,z),
\]
where the last equality holds from $\omega(x,e_1)=0$ for all $x\in\hg$. Hence, $\omega _\mathfrak u (v,z)=0$ for all $z\in \mathfrak u$, which implies that $v=0$ since 
$\omega _\mathfrak u$ is non-degenerate. Finally, since $\omega$ is closed,  $\mathfrak u$ is abelian and $\omega(x,e_1)=0$, for all $x\in\mathfrak u$, and  it follows that  $\Pi \circ \left(\text{ad}_{e_0}|_\mathfrak u \right)\in \spg (\mathfrak u, \omega_\mathfrak{u})$. 
\end{proof}

\smallskip

\begin{remark}\label{remark:symp}
According to Proposition \ref{prop:symp}, given a symplectic structure $\omega$ on a $2n$-dimensional almost abelian Lie algebra $\ggo$,  there exists  a basis $\{e_0, e_1,\ldots,e_{2n-1}\}$ of $\ggo$ such that $\omega(e_0,e_1)=1$, $\hg=\text{span}\, \{e_i : i\geq 1\}$ and $\mathfrak u =\text{span}\, \{e_i : i\geq 2\}$. Let $\Omega$ be the matrix of $\omega_{\mathfrak u}$ in the given basis of $\mathfrak u$, then
the matrix $C$ of $\ad_{e_0}|_\hg$ in the basis $\{ e_i : i\geq 1\}$ is given by
\begin{equation}\label{sympl}
    C=\begin{pmatrix} 
c & w^T  \\
0 &  D 
\end{pmatrix}, \qquad \; \text{where } c \in \RR,\,  w \in \RR^{2(n-1)},\; D\in \spg_{\Omega}(2(n-1),\RR). 
\end{equation} 

Conversely, let $C$ be a matrix  as in \eqref{sympl}  where $\Omega\in M(2(n-1),\RR)$ is a non-singular skew-symmetric matrix. Consider the $2n$-dimensional almost abelian Lie algebra $\ggo_C= \RR e_0 \ltimes _C \RR^{2n-1}  $, that is, $[e_0, e_i]=Ce_i$, $i\geq 1$.   
Let $\mathfrak u =\text{span}\, \{e_i : i\geq 2\}$ and define  $\omega:=e^0\wedge e^1+ \omega_{\mathfrak u}$, where $\omega_{\mathfrak u}$ is the non-degenerate $2$-form on $\mathfrak u$ whose matrix in the given basis of $\mathfrak u$ is  $\Omega$. It is easy to see that $\omega$ is a non-degenerate and closed $2$-form on $\ggo_C.$ 
\end{remark}
\begin{corollary}
    \label{symp-v columna}
 Let $\ggo_A$ be a $2m$-dimensional almost abelian Lie algebra. Then $\ggo_A$   admits a symplectic structure if and only if $A$ is conjugate to 
\begin{equation} \label{eq:symp-matrix}
C= \begin{pmatrix}
      c   & 0 \\
      w   & D
\end{pmatrix}, \quad \text{where } c\in \RR , \,  w \in \mathbb R ^{2(m-1)}, \; D\in \spg_{\Omega}(2(m-1), \RR),   
\end{equation}
for some non-singular skew-symmetric $\Omega\in M(2(m-1), \mathbb R )$. 
\end{corollary} 

\begin{proof}
According to Remark \ref{remark:symp},  $\ggo_A$ admits a symplectic structure if and only if $A$ is conjugate to 
$$
C= \begin{pmatrix}
      c   & v^T \\
      0   & B
\end{pmatrix}, \quad \text{where }c\in \RR , \,  v \in \mathbb R ^{2(m-1)}, \; B\in \spg_{\Omega_1}(2(m-1), \RR),
$$
with $\Omega_1=P^{T}J_m P$, for some $P\in GL(2(m-1), \RR)$. Since $C$ is conjugate to its transpose, then $A$ is conjugate to $
 \begin{pmatrix}
      c   & 0 \\
      v   & B^T
\end{pmatrix}$ where $B^T\in \spg_{\Omega_2}(2(m-1), \RR)$ with $\Omega_2=P^{-1}J_m (P^{-1})^T$ (see Remark \ref{rem-symp-conj}). The corollary follows by taking $\Omega= \Omega_2$.
\end{proof}
There is a classical theorem describing nilpotent orbits of the Lie algebra $ \spg (2m, \RR)$ (see \cite[Chapter 5]{CMcG}). 
This result was generalized in \cite[Theorem 2.4]{Mehl}  to any orbit of $ \spg (2m, \RR)$. The  next theorem, which will be useful for our purposes, is a consequence of  \cite[Theorem 2.4]{Mehl}.        
\begin{theorem}\label{prop:symplex-nonzero}
Let $\Omega\in M(2m, \RR)$ be a non-singular skew-symmetric matrix and let  $B\in M(2m, \RR)$. Then $B$  is conjugate in $GL(2m, \RR)$ to a matrix in  $ \spg _{\Omega} (2m, \RR)$ if and only if any $\lambda \in \emph{spec}\, (B)$ satisfies:
 \begin{align}
 \label{symp-real-eig} &\text{if }\lambda =0 \text{ and } \Lambda (B(0))=( n_1,\ldots , n_k; p_1, \ldots , p_k; t), \text{ then }  p_r \text{ is even for all odd } n_r\text{ and } t \text{ is}\\
\nonumber &  \text{even};\\
  \label{symp-complex-eig}  &  \text{if }  \lambda \notin i\RR , \text{ then  } -\lambda \in\emph{spec}\, (B)  \text{ and } \Lambda (B(-\lambda))=\Lambda (B(\lambda)).
  \end{align}
In particular, if $B$  is conjugate  to a matrix in  $ \spg _{\Omega} (2m, \RR)$ then its semisimple and nilpotent parts are conjugate to matrices in $ \spg _{\Omega_1} (2m, \RR)$ and $ \spg _{\Omega_2} (2m, \RR)$, respectively, for some $\Omega_1$ and $\Omega_2$.
  \end{theorem}
  \begin{proof} 
  The first part of the theorem follows from \cite[Theorem 2.4]{Mehl}. 
  
  To prove the last assertion, decompose $B$  as $B=B_s+B_n$ with $B_s$ semisimple and $B_n$ nilpotent. It is easily verified that if $B$ satisfies conditions \eqref{symp-real-eig} and \eqref{symp-complex-eig}, then so do $B_n$ and $B_s$, hence  the last assertion follows.
  \end{proof}

\subsection{The nilpotent case}\label{sec:nilp-symp}
The next well-known lemma (see for instance, \cite[Theorem 5.1.3]{CMcG}) will be useful to prove existence of symplectic structures on certain almost  abelian Lie algebras. 
\begin{lemma}\label{J2k} The matrix $\J_{2k}$ is conjugate to a matrix in $\spg (2k, \RR)$ for all $k\geq 1$.
    \end{lemma}
    \begin{proof} Consider the following nilpotent matrix:
    \[ N=\begin{pmatrix}
        X& I_k\\
        0_k& -X^T
    \end{pmatrix}\in \spg (2k, \RR), \quad X= \J _k.
    \]
  
    We will show that $N^{2k-1}\neq 0$, or equivalently, $N$ is conjugate to $\J_{2k}$. 
    
    In order to prove that $N^{2k-1}\neq 0$, we compute $N^r$ for $r\geq 2$. It can be shown inductively that 
    \[N^r =\begin{pmatrix} X^r & B_r\\
    0_k & (-X^T)^r        
    \end{pmatrix},\;\; \text{where } B_r=\sum _{i=1}^r (-1)^{i-1} X^{r-i}(X^T)^{i-1}.  \]
    For $r=2k-1$, the  terms in the expression of $B_{2k-1}$ corresponding to $i\neq k$ vanish. Hence, $B_{2k-1}=(-1)^{k-1} X^{k-1}(X^T)^{k-1}$, which is the matrix whose $(k,k)$ coefficient is  $(-1)^{k-1} $ and the remaining entries are equal to zero. In particular,   $B_{2k-1}\neq 0$, which implies that  $N^{2k-1}\neq 0$. Therefore, $N$ is conjugate to $\J_{2k}$, and the lemma follows.
    \end{proof}
Let $\ggo=\RR e_0 \ltimes \hg $ be an even dimensional almost abelian Lie algebra. We prove next  that if $\ggo$ 
is $k$-step nilpotent   and $k=2$, $3$ or $4$, then $\ggo$  always admits a symplectic structure. This is no longer true for $k=5$, as Example \ref{ex:5-step-no-symp} below shows. 

\begin{theorem} \label{thm:k-step-symp} Let $\ggo_A=\RR e_0 \ltimes_A \RR^{2n-1} $ be a $k$-step nilpotent almost abelian  Lie algebra with  $k=2, \, 3$ or $4$.  Then $\ggo_A$ admits a symplectic structure. 
\end{theorem}

\begin{proof}
Throughout the proof, we will use several times Remark \ref{suma w_i} and the facts that $(\J_3 \oplus -\J_3^T) \in \spg \left(6, \RR\right)$, $\J_2 \in \spg \left(2, \RR\right)$, $0_{2k} \in \spg \left(2k, \RR\right)$ (see \eqref{eq:sympl}) and $\J_4 \in \spg(4,\RR)$ (by Lemma~ \ref{J2k}). 

Assume first that $k=2$, hence, the minimal polynomial of $A$ is $x^2$ and the Jordan normal form of  $A$ is  $\J_2  ^{ \oplus r}\oplus 0_s $, for $r,\, s>0$. It follows that $s=2k+1$ with $k= n-r-1$. 

If $r$ is even, then 
by a change of basis, we obtain that $A$ is conjugate to $C= 0_1\oplus \J_2 ^{\oplus r}\oplus 0_{2k}$, and clearly this has the form \eqref{sympl} with $c=0$, $w=0$ and $D=\J_2 ^{\oplus r}\oplus 0_{2k}\in \spg_{\Omega} \left(2(n-1), \RR\right)$, where $\Omega=J_1^{\oplus r}\oplus J_k$.

If $r$ is odd, then 
by a change of basis, we obtain that $A$ is conjugate to $C= \J_2^T \oplus 0_{2k+1}\oplus \J_2 ^{\oplus (r-1)}$, which has the form \eqref{sympl} since
\[
C=\begin{pmatrix}
    0 & \epsilon_1 ^T\\
    0 & D
\end{pmatrix},
\] with $D= 0_{2k+2}\oplus \J_2 ^{\oplus (r-1)} \in \spg_{\Omega} \left(2(n-1), \RR\right)$, where $\Omega=J_{k+1}\oplus J_1 ^{\oplus (r-1)}$.  

If $k=3$, the Jordan normal form of  $A$ is  $\J_3  ^{ \oplus q}\oplus\J_2  ^{ \oplus r}\oplus 0_s $, for $q>0$.

If $q$ is even, $q=2l$, then $s$ is odd, $s=2k+1$. Therefore, $A$ is conjugate to $C = 0_1 \oplus (\J_3 \oplus -\J_3^T)^{ \oplus l}\oplus\J_2  ^{ \oplus r}\oplus 0_{2k}$, which has the form \eqref{sympl} with $c=0$, $w=0$ and $D=(\J_3 \oplus -\J_3^T)^{ \oplus l}\oplus\J_2  ^{ \oplus r}\oplus 0_{2k} \in \spg_{\Omega} \left(2(n-1), \RR\right)$, where $\Omega=J_{3}^{\oplus l}\oplus J_1 ^{\oplus r}\oplus J_k$. 

If $q$ is odd, $q-1=2l$, then $s$ is even, $s=2k$. Then, $A$ is conjugate to $C= \J_3^T \oplus\J_2  ^{ \oplus r} \oplus (\J_3 \oplus -\J_3^T)^{ \oplus l}\oplus 0_{2k}$ and this has the form \eqref{sympl}:
\[
C=\begin{pmatrix}
    0 & \epsilon_1 ^T\\
    0 & D
\end{pmatrix}
\] 
where 
$D= \J_2  ^{ \oplus (r+1)}\oplus (\J_3 \oplus -\J_3^T)^{ \oplus l}\oplus 0_{2k}\in \spg_{\Omega} \left(2(n-1), \RR\right)$, where $\Omega=J_1^{\oplus (r+1)}\oplus J_3^{ \oplus l}\oplus J_{k}$. 

Finally, if $k=4$,  the Jordan normal form of  $A$  is  
$
\J_4^{\oplus p}\oplus \J_3^{\oplus q} \oplus \J_2^{\oplus r}\oplus 0_s$, for $p>0$.

Since $A \in M(2n-1,\RR)$ it follows that either
$q$ is  even or $s$ is even, but not both at the same time.

If $q$ is even, $q=2l$, then $s$ is odd, $s=2k+1$. Therefore, $A$ is conjugate to $C = 0_1 \oplus\J_4  ^{ \oplus p}\oplus (\J_3 \oplus -\J_3^T)^{ \oplus l}\oplus\J_2  ^{ \oplus r}\oplus 0_{2k}$, which has the form \eqref{sympl} with $c=0$, $w=0$ and $D=\J_4  ^{ \oplus p}\oplus (\J_3 \oplus -\J_3^T)^{ \oplus l}\oplus\J_2  ^{ \oplus r}\oplus 0_{2k} \in \spg_{\Omega} \left(2(n-1),\RR\right)$, where $\Omega=J_2 ^{\oplus p}\oplus J_3 ^{\oplus l}\oplus J_1^{\oplus r}\oplus  J_k$.

If $q$ is odd, $q=2l+1$, then $s$ is even $s=2k$. Then, $A$ is conjugate to $C= \J_3^T \oplus  \J_2 ^{ \oplus r} \oplus\J_4  ^{ \oplus p} \oplus (\J_3 \oplus -\J_3^T)^{ \oplus l}\oplus 0_{2k}$ and this take the form \eqref{sympl}:
\[
C=\begin{pmatrix}
    0 & \epsilon_1 ^T\\
    0 & D
\end{pmatrix}
\] 
with $D=
\J_2^{\oplus r+1}\oplus  \J_4  ^{ \oplus p} \oplus   (\J_3 \oplus -\J_3^T)  ^{ \oplus l}\oplus 0_{2k} \in \spg_{\Omega} \left(2(n-1),\RR\right)$, where $\Omega=J_1^{\oplus (r+1)}\oplus  J_2 ^{\oplus p}\oplus J_3 ^{\oplus l}\oplus  J_k$.
\end{proof}

We show below  an example of   a 5-step nilpotent almost abelian Lie algebra  which does not admit a symplectic structure. 
\begin{example}\label{ex:5-step-no-symp}
    Let $\ggo_M=\RR e_0 \ltimes_M \RR^9 $ be the $10$-dimensional almost abelian Lie algebra with  $M=\J_5  \oplus \J_3  \oplus 0_1$. We claim that $\ggo_M$ does not admit a symplectic structure. Indeed, let $\mathcal B=\{e_0,\dots, e_9\}$ be the basis of $\ggo _M$ such that $M$ is the matrix of  $\ad_{e_0}|_{\RR^9}$ with respect to $\{e_1,\dots, e_9\}$ and denote by $\{e^0, \dots , e^9\}$ the dual basis of $\mathcal B$. 
  
    We have that, 
$$
\begin{cases}
\begin{array}{cccc}
 d e^k &= &-e^0 \wedge e^{k+1}, & \text{for } k=1, \ldots, 7,  \; k \neq 5, \\
de^k &= &0, & \text{for } k= 0, 5, 8,9.
\end{array}
\end{cases}
$$
Hence, any closed 2-form on $\ggo$ can be written as 
$ \omega = \sum _{j=1} ^9 a_j e^{0j} + \beta$, where $\beta$ is given by:
$$
\beta= b_1 \,e^{25} - b_1 \,e^{34} + b_2\, e^{38} + b_3\,  e^{45} - b_2 \,e^{47} +  b_4\, e^{48} + 
b_2 \,e^{56} - b_4\,  e^{57} + b_5 \, e^{58} + b_6 \, e^{59} + b_7 \, e^{78} + b_8 \, e^{89},
$$
with $a_j, b_k \in \RR$, for $j=1, \ldots, 9$ and $k=1, \ldots, 8$.  Here, $e^{ij}$  denotes the 2-form $e^i \wedge e^j$.

It follows that all these 2-forms are degenerate. In fact, if $a_1 = 0$ we have that $\omega (e_1, \cdot) \equiv 0$. If $a_1 \neq 0$ and $b_1 \neq 0$, there exists $x= -\dfrac{a_2 b_2 + a_6 b_1 }{a_1 b_1}\, e_1 + \dfrac{b_2}{b_1}\, e_2 + e_6 \neq 0$  such that $\omega (x, \cdot) \equiv 0$. Finally, if $b_1 =0$, $x= -\dfrac{a_2}{a_1}\, e_1 + e_2\neq 0$ satisfies $\omega (x, \cdot) \equiv 0$.  
\end{example}
Given a nilpotent matrix $M$, in the following theorem we give necessary and sufficient conditions on the tuple $\Lambda (M)$ to obtain pairwise non-isomorphic nilpotent almost abelian Lie algebras $\ggo_M$ admitting symplectic structures. We exclude the zero matrix $M$, since when $M=0$, the Lie algebra $\ggo_M$ is  abelian and it  admits a symplectic structure. 

\begin{theorem}\label{classifsymp} 
Let $\ggo_M=\RR e_0  \ltimes_M \RR^{2n-1}$ be the   $2n$-dimensional  nilpotent almost abelian Lie algebra with $\Lambda(M)=(n_1,\ldots, n_k; p_1, \ldots, p_k;t)$. Then  $\ggo_M$ admits a symplectic structure if and only if one of the following conditions holds:

\begin{enumerate} 
\item[$\ri$] $t$ is odd and $p_i $ is even for all odd $n_i$;
\item[$\rii$] $t$ is even and there exists a unique $\ell \in \{1, \dots, k\}$ such that $n_\ell$ and $p_\ell$ are odd.
\end{enumerate}
\end{theorem}

\begin{proof}
We suppose that $\ggo_M$ admits a symplectic structure. Hence, by Corollary \ref{symp-v columna}, $M$ is conjugate to     $C= \begin{pmatrix}
0   & 0 \\
w  & D
   \end{pmatrix}$, where $w \in \mathbb R ^{2(n-1)}$ and $D\in \spg(2(n-1), \RR)$. 
   Since $C$ is a non-zero nilpotent matrix, we have that,  either $D=0$ and $w\neq 0$ or $D$ is a nonzero nilpotent matrix. If $D=0$, Corollary \ref{posibles_C} implies that $C$, hence $M$, is conjugate to $\J_2\oplus 0_{2n-3}$. Therefore, $M$ satisfies condition (i).
   Now, we suppose that $D$ is a nonzero nilpotent matrix.    

   Then, 
   Theorem \ref{prop:symplex-nonzero} implies that $D$ is conjugate to 
    $$
    \J _{m_1}^{\oplus q_1}  \oplus \cdots \oplus \J _{m_r}^{\oplus q_r} \oplus  0_{s}
    $$ 
    with $m_1 > \cdots > m_r\geq 2$, $q_i \geq 1$, for all $i=1, \ldots, r$, $s\geq 0$ and where $s$ is even and $q_i$ is even for all $i$ such that $m_i$ is odd. Applying Corollary \ref{posibles_C} we obtain that $C$ (in particular, $M$) is conjugate to one of the three possibilities:
\begin{enumerate}[(a)]
        \item $0_1 \oplus D$;
        \item $\J _{m_1}^{\oplus q_1}  \oplus \cdots \oplus \J_{m _{d-1}}^{\oplus q_{d-1}} \oplus \J_{m_{d} + 1}\oplus \J_{m_{d }}^{\oplus q_{d}-1} \oplus\J_{m_{d +1}}^{\oplus q_{d+1}} \oplus \cdots \oplus  \J _{m_r}^{\oplus q_r} \oplus  0_s$, for $d \in \{ 1, \dots, r\}$;
        \item $\J _{m_1}^{\oplus q_1}  \oplus \cdots \oplus \J _{m_r}^{\oplus q_r} \oplus \J _{2}\oplus  0_{s-1}$.
    \end{enumerate}      
    
By  uniqueness of the Jordan normal form, cases (a) and (c) imply that  condition (i) of the theorem is satisfied.   

Next we consider that $M$ is conjugate to the matrix given in the case (b). In this case  $t$ is even, since $t=s$.
We suppose that $m_d$ is even, so $m_d +1$ is odd and $m_{d-1} \geq m_{d} + 1$. If $m_{d-1} > m_{d} + 1$, then $\J_{m_{d} + 1}$ appears only once. 
If $m_{d-1}=m_{d} + 1$, then $m_{d-1}$ is odd. So Theorem \ref{prop:symplex-nonzero} implies that $q_{d-1}$ is even. In this way, $\J_{m_{d} + 1}$ appears an odd number of times. Now, we suppose that $m_d$ is odd. By Theorem \ref{prop:symplex-nonzero}, $q_d$ is even and Corollary \ref{posibles_C} says that  $\J_{m_{d}}$ appears $q_d - 1$ times, that is, an odd number of times. Finally, the uniqueness in condition (ii) follows from  Theorem \ref{prop:symplex-nonzero}.   

\vspace{0.25cm}

Conversely, suppose first that $M$ satisfies condition (i). Then $M$ is conjugate to  
$\begin{pmatrix} 
0 & 0 \\
0 & D
\end{pmatrix}$, where 
$$
D= 0_{t-1}\oplus \J _{n_1}^{\oplus p_1}  \oplus \cdots \oplus \J _{n_k}^{\oplus p_k},
$$
$t-1$ is even and $p_i$ is even for all $i$ such that $n_i$ is odd. Applying Theorem \ref{prop:symplex-nonzero}, $D$ is conjugate to a matrix in $ \spg(2(n-1), \RR)$. By Corollary \ref{symp-v columna}, $\ggo_M$ admits a symplectic structure.  

Next, we suppose that condition (ii) is satisfied. Let  $\ell \in \{1, \dots, k\}$ be the unique index  such that $n_\ell$ and $p_\ell$ are odd. We have that $M$ is conjugate to 
$$
C=\J _{n_\ell}\oplus \J _{n_1}^{\oplus p_1}  \oplus \cdots \oplus  \J _{n_\ell - 1}^{\oplus p_{\ell -1}}\oplus \J _{n_{\ell}}^{\oplus p_\ell -1}\oplus  \J_{n_{\ell+1}}^{\oplus p_{\ell +1}}\oplus\cdots \oplus \J _{n_k}^{\oplus p_k}\oplus 0_{t}.  
$$
So, $C=\begin{pmatrix}
0   & 0 \\
\epsilon_1 & D
\end{pmatrix}$ where 
$$
D= \J _{{n_\ell}-1}\oplus \J _{n_1}^{\oplus p_1}  \oplus \cdots  \oplus  \J _{n_{\ell-1}}^{\oplus p_{\ell -1}}\oplus \J _{n_{\ell}}^{\oplus p_\ell -1}\oplus  \J_{n_{\ell+1}}^{\oplus p_{\ell +1}}\oplus\cdots \oplus \J _{n_k}^{\oplus p_k}\oplus 0_{t}.  
$$
Since $n_\ell -1$, $t$  and $p_\ell -1$ are even, by Theorem \ref{prop:symplex-nonzero} we obtain that $D$ is conjugate to a matrix in $\spg(2(n-1), \RR)$,  and again, by Remark \ref{remark:symp}, $\ggo_M$ admits a symplectic structure.  \end{proof}

\ 

For each $n\in \NN$, let 
\begin{equation*}
\begin{split}  
    {\bigwedge}^s_{2n-1}:&=\{ (n_1,\ldots , n_k; p_1, \ldots , p_k; t ) \in {\bigwedge}_{2n-1}:  \text{satisfying (i) or (ii)} \\ & \qquad \text{ of Theorem \ref{classifsymp}} \}\cup\{(2n-1)\}.
\end{split} 
\end{equation*}

Consider the bijection  $\tilde{\Lambda}$ 
defined in \eqref{eq:bijection-tuples} for $b=0$ and $m=2n-1$. As a consequence of Lemma \ref{lem:isom-nilpot} and Theorem \ref{classifsymp},  we obtain the following classification result.
\begin{corollary}
\label{cor:classif-symp} Let $n\in \NN$. The isomorphism classes of $2n$-dimensional   nilpotent almost abelian Lie algebras admitting a symplectic  structure are parametrized by the set $\bigwedge^s_{2n-1}$. \end{corollary}
The next corollary follows from the fact  that $\bigwedge^s_7=\bigwedge_7$. This is in contrast with the complex case (see Example \ref{ex:complex-8}).
\begin{corollary}\label{sym8}
Any 8-dimensional nilpotent almost abelian Lie algebra admits a symplectic structure.
\end{corollary}
As a consequence of Theorem \ref{classifcomplex} and Theorem \ref{classifsymp} we obtain the next result, whose proof follows from the straightforward fact that $\bigwedge^c_{2n-1} \subset \bigwedge^s_{2n-1}$.

\begin{theorem}\label{complex-sympl}
  Let $\ggo_A=\RR e_0  \ltimes_A \RR^{2n-1}$ be a nilpotent almost abelian Lie algebra. If $\ggo_A$ admits a complex structure, then $\ggo_A$ admits a symplectic structure.   
\end{theorem}

In Example \ref{ex:5-step-no-symp} we 
showed by direct computation that the nilpotent almost abelian Lie algebra corresponding to the tuple $(5,3;1,1;1)$  admits no symplectic structure. It is easy to check that this is   the only tuple in ${\bigwedge}_{9} - {\bigwedge}_9^s$. The isomorphism classes  of 10-dimensional nilpotent almost abelian Lie algebras admitting neither complex nor symplectic structure are parametrized by the following intersection:
\[ \left({\bigwedge}_{9} - {\bigwedge}_9^c\right)\bigcap  \left({\bigwedge}_{9} - {\bigwedge}_9^s\right)= {\bigwedge}_{9} - {\bigwedge}_9^s=\{(5,3;1,1;1)\}. \]
This implies the following corollary.
\begin{corollary}\label{sym10}
  Up to isomorphism, 
the only $10$-dimensional nilpotent almost abelian Lie algebra  admitting neither complex  nor symplectic structure is $\ggo_M$ with $\Lambda(M)=(5,3;1,1;1)$.  
\end{corollary}

\begin{remark} We point out that if $\ggo_M$ is a nilpotent almost abelian Lie algebra admitting no symplectic structure, then $\dim \ggo _M \geq 10$. In fact, it follows from Theorem \ref{thm:k-step-symp} that if $\ggo _M$ is $k$-step nilpotent with $k\leq 4$, then $\ggo _M$ admits a symplectic structure. Therefore, if $\ggo _M$ admits no symplectic structure we must have $k\geq 5$. 
It was shown in \cite{GK} that the Lie algebra $\ggo_M$ with $\Lambda(M)=(5;1;0)$ admits a symplectic structure. 
This forces the size of the matrix $M$ to be $m\times m$ with $m\geq 7$, that is, $\dim \ggo _M \geq 8$. Since all $8$-dimensional nilpotent almost abelian Lie algebras admit a symplectic structure (see Corollary \ref{sym8}), it follows that $\dim \ggo _M \geq 10$.
\end{remark}

\begin{example} Corollary \ref{sym10} can be generalized by  observing that, for $m_1$ and $m_2$ odd, with $m_1>m_2$, the tuple $(m_1,m_2;1,1;1)\in \bigwedge _{m\, }- \bigwedge _{m\, }^s$, for $m=m_1+m_2+1$. Therefore, the corresponding 
nilpotent almost abelian Lie algebra admits no symplectic structure. Moreover, by Theorem \ref{complex-sympl}, this family of Lie algebras admits neither complex nor symplectic structure (see Corollary \ref{cor:final} for other examples). \end{example}
Theorem \ref{complex-sympl} does not hold for arbitrary nilpotent Lie algebras, as the next example shows.

\begin{example}\label{ex_symp-no-cx}
    Let $\ggo=\RR \oplus \hg_5$, where $\hg_5$ is the $5$-dimensional Heisenberg Lie algebra. Then $\ggo$ admits a complex structure but no symplectic structure (see \cite{Sal} and \cite{GK}, respectively).  
\end{example}
\subsection{Applications of Theorem \ref{classifsymp}}

In Theorem \ref{thm:k-step-symp} we have seen that every $k$-step nilpotent almost abelian  Lie algebra, with  $k=2,3$ or $4$, admits a symplectic structure. However, we have seen that this fails for 5-step nilpotent almost abelian Lie algebras if we consider Lie algebras of dimension 10. We study next the case of $6$-step nilpotent almost abelian Lie algebras.

\begin{corollary}\label{6-step nilp}
Let $\ggo_A$ be a 6-step nilpotent almost abelian Lie algebra of dimension $2n$. If $2n\leq 14$, then $\ggo_A$ admits a symplectic structure. 
\end{corollary}

\begin{proof}
If $\ggo_A$ is a 6-step nilpotent almost abelian Lie algebra of dimension $2n$, then $A$ is conjugate to  its Jordan normal matrix $M$ given by
    \[ M= \J _{6}^{\oplus p_6}  \oplus \J _{5}^{\oplus p_5}  \oplus \J _{4}^{\oplus p_4}  \oplus\J _{3}^{\oplus p_3}  \oplus \J _{2}^{\oplus p_2}  \oplus  0_{t},  \quad  p_6\geq 1, \;\;  t\geq 0,\quad \sum_{i=2}^6 ip_i + t = 2n-1 , \]
where $p_i\geq 0$ for $i=2, \ldots ,5$. The meaning of $p_i=0$ is that the Jordan block $\J_i$ does not appear in $M$. 
Since $\sum_{i=2}^6 ip_i + t$ is odd, it follows that $5p_5+3p_3 +t$ is odd. If $t$ is even, then $p_5+p_3$ is odd, so only one of $p_5$ and $p_3$ is odd, then condition $\rii$ of Theorem \ref{classifsymp} is satisfied and $\ggo_M$ admits a symplectic structure. So, by Remark \ref{rem:nilp-similiar}, $\ggo_A$ also admits a symplectic structure.   If $t$ is odd ($t\geq1$), then $p_5+p_3$ is even, so $p_5$ and $p_3$ are both odd or both even. Suppose $p_5$ and $p_3$ are both odd ($p_5\geq 1$ and $p_3\geq 1$), we have 
\[ 15 =6+5+3+1 \leq \sum_{i=1}^6 ip_i + t = 2n-1\leq 14-1=13, \]
which is a contradiction. Therefore $p_5$ and $p_3$ are both even, condition (i) of Theorem \ref{classifsymp} is satisfied and $\ggo_M$ admits a symplectic structure. Using again Remark \ref{rem:nilp-similiar}, we obtain that $\ggo_A$ admits a symplectic structure.  
\end{proof}

The above result can be generalized as follows.

\begin{corollary}\label{even-step}
  Let $\ggo_A$ be a k-step nilpotent almost abelian Lie algebra where  $k$ is even and  $k\geq 6$. If $\dim  \ggo_A=2n\leq k+8$, then $\ggo_A$ admits a symplectic structure.  
\end{corollary}
\begin{proof}
If $\ggo_A$ is a $k$-step nilpotent almost abelian Lie algebra of dimension $2n$, then $A$ is conjugate to its Jordan normal form $M$ given by

    \[ M= \J _{k}^{\oplus p_k}  \oplus\cdots \oplus\J _{4}^{\oplus p_4}  \oplus\J _{3}^{\oplus p_3}  \oplus \J _{2}^{\oplus p_2}  \oplus  0_{t},  \quad  p_k\geq 1, \;\;  t\geq 0,\quad \sum_{i=2}^k ip_i + t = 2n-1,\]
where $p_i\geq 0$ for $i=2, \ldots ,k-1$. The meaning of $p_i=0$ is that the Jordan block $\J_i$ does not appear in $M$. 
Since $p_k\geq 1$ we have 
\[ k + k(p_k-1) + \sum_{i=2}^{k-1} ip_i + t=\sum_{i=2}^k ip_i + t =2n-1 \leq k+7 .\]
Then $ k(p_k-1) + \sum_{i=2}^{k-1} ip_i + t $ is odd and less than or equal to $7$.

By hypothesis $k=2l$ with $l>2$.

If $t$ is odd, $ k(p_k-1) + \sum_{i=2}^{k-1} ip_i$ is even. In particular, $\displaystyle\sum_{i=2}^{l} (2i-1)p_{2i-1}$ is even. Hence, either there is  an even number of indices $i\in \{ 2, \ldots, l\}$ such that $p_{2i-1}$ is odd,  or $p_{2i-1}$ is even for all $i=2, \ldots, l$. Suppose that the first condition holds. So, there exist at least two indices, say $i_1 > i_2$, such that $p_{2i_1 -1},\, p_{2i_2-1}$ are odd. Since $2i_1-1 \geq 5$ and $2i_2-1 \geq 3$,  we obtain
\begin{equation}\label{i1i2}
8 \leq (2i_1-1)+(2i_2-1)\leq k(p_k-1) + \sum_{i=2}^{k-1} ip_i\leq 7,    
\end{equation}
which is a contradiction. Then, condition (i) of Theorem \ref{classifsymp} is satisfied and $\ggo_M$ admits a symplectic structure.

If $t$ is even, $ k(p_k-1) + \sum_{i=2}^{k-1} ip_i$ is odd. In particular,  $\displaystyle\sum_{i=2}^{l} (2i-1)p_{2i-1}$ is odd. Hence, there exists an odd number of indices  $i\in \{ 2, \ldots, l\}$ such that $p_{2i-1}$ is odd. But, if there exist at least two different numbers $i_1, i_2$ satisfying this condition, the contradiction given in \eqref{i1i2} is obtained again. Therefore, condition $\rii$ of Theorem \ref{classifsymp} is satisfied and $\ggo_M$ admits a symplectic structure.  Using again Remark \ref{rem:nilp-similiar}, we obtain that $\ggo_A$ admits a symplectic structure.  
\end{proof}

The upper bound $k+8$ in Corollary \ref{even-step} is optimal, since if we consider  even  $k\geq6$ and $M=\J_k\oplus \J_5\oplus \J_3\oplus 0_{1}$, then $\ggo_M$ is a $k$-step nilpotent almost abelian Lie algebra of dimension $k+10$, but by Theorem \ref{classifsymp} $\ggo_M$ does not admit a symplectic structure.

\begin{corollary}\label{odd-step}
 Let $\ggo_A$ be a k-step nilpotent almost abelian Lie algebra  where  $k$ is odd and $k\geq5$. If $\dim \ggo_A = 2n\leq k+3$, then $\ggo_A$ admits a symplectic structure.  
\end{corollary}

\begin{proof}
If $\ggo_A$ is a $k$-step nilpotent almost abelian Lie algebra of dimension $2n$, then $A$ is conjugate to its Jordan normal form $M$ given by
    \[ M= \J _{k}^{\oplus p_k}  \oplus\cdots \oplus\J _{4}^{\oplus p_4}  \oplus\J _{3}^{\oplus p_3}  \oplus \J _{2}^{\oplus p_2}  \oplus  0_{t},  \quad  p_k\geq 1, \;\;  t\geq 0,\quad \sum_{i=2}^k ip_i + t = 2n-1 ,\]
where $p_i\geq 0$ for $i=2,\ldots , k-1$. As before, the meaning of $p_i=0$ is that the Jordan block $\J_i$ does not appear in $M$. 
Since $p_k\geq 1$ we have 
\[ k + k(p_k-1) + \sum_{i=2}^{k-1} ip_i + t=\sum_{i=2}^k ip_i + t =2n-1 \leq k+2 .\]
Then $ k(p_k-1) + \sum_{i=2}^{k-1} ip_i + t $ is even and less than or equal to $2$.
Therefore, $M=\J_k$, $M=\J_k \oplus \J_2$ or $M=\J_k \oplus 0_2$, any of the three possibilities admits symplectic structure by Theorem \ref{classifsymp}.  Using again Remark \ref{rem:nilp-similiar}, we obtain that $\ggo_A$ admits a symplectic structure.  
\end{proof}

The upper bound $k+3$ in Corollary \ref{odd-step} is optimal, since if we consider an odd integer   $k\geq5$ and $M=\J_k\oplus \J_3\oplus 0_{1}$, then $\ggo_M$ is a $k$-step nilpotent almost abelian Lie algebra of dimension $k+5$, but by Theorem \ref{classifsymp} $\ggo_M$ does not admit a symplectic structure.

As a final application of Theorem \ref{classifcomplex} and Theorem \ref{classifsymp}, we obtain families of nilpotent almost abelian Lie algebras admitting both, complex and symplectic structures, or   symplectic but no complex structure, or  neither complex nor symplectic structure.
\begin{corollary}\label{cor:final} Let $\ggo_M$ be a nilpotent almost abelian Lie algebra.
\begin{enumerate}
    \item[$\ri$] If $M=\J_{m+1}\oplus \J_m$, then $\ggo_M$ admits both, complex and symplectic structure for all $m\geq 2.$
    \item[$\rii$] If $M=\J_m$ with $m$ odd, then $\ggo_M$ admits symplectic but no complex structure.
    \item[$\riii$] If $M=\J_{m_1}\oplus \J_{m_2}\oplus \J_{m_3}$ with $m_1$, $m_2$ and $m_3$ odd, $m_1>m_2>m_3$, then $\ggo_M$ admits neither complex nor symplectic structure. 
\end{enumerate}
\end{corollary}

\subsection{The general case}
We start by proving a result that will be needed in Theorem \ref{sympM(c)+M(-c)} and Theorem \ref{sympM(c)+M(-c)+Q}.
\begin{proposition}
\label{odd-space sum}
Let $M\in M(2n-1, \RR )$ and assume that $\ggo_M$ admits a symplectic structure, then one only one of the following conditions holds:
\begin{enumerate}
        \item[\emph{(I)}] $0 \in \emph{spec}_\RR (M)$ and $m_0$ is odd; 
        \item[\emph{(II)}] there exists a unique $c\in \emph{spec}_\RR (M)$, $c\neq 0$, such that $-c\in \emph{spec}_\RR (M)$ with $m_{-c}=m_c -1$;
         \item[\emph{(III)}] there exists a unique $c\in \emph{spec}_\RR (M)$, $c\neq 0$,  such that $m_c=1$ and $-c\notin \emph{spec}_\RR (M)$.
\end{enumerate}
\end{proposition}
\begin{proof}
By Corollary \ref{symp-v columna}, $M$ is conjugate to 
$C= \begin{pmatrix}
      c   & 0 \\
      w   & D
     \end{pmatrix}$ where $c\in \RR$, $w\in \RR^{2n-2}$
     and $D\in \spg(2(n-1), \RR)$. Since $M$ and $C$ are conjugate and $c$ is an eigenvalue of $C$, then $c\in \text{spec}_\RR (M)$. 
     If $c\neq 0$ and $c\notin \text{spec}_\RR (D)$, then $m_c=1$ and, by \eqref{symp-complex-eig} of Theorem \ref{prop:symplex-nonzero}, $-c\notin \text{spec}_\RR (M)$. Now, suppose that $c \in \text{spec}_\RR (D)$. If $c=0$, by condition \eqref{symp-real-eig} of Theorem \ref{prop:symplex-nonzero} we obtain that $m_c$ is odd.  On the other hand, if $c\neq 0$ by condition \eqref{symp-complex-eig} of Theorem \ref{prop:symplex-nonzero} we have that $-c \in \text{spec}_\RR (D)$ with $\Lambda (D(c))=\Lambda (D(-c))$. Here, $D$ is conjugate to $D(c) \oplus D(-c) \oplus \tilde{D}$, where  $\pm c\notin \text{spec}_\RR (\tilde{D})$). Hence, $-c \in \text{spec}_\RR (M)$ and $m_{-c}=m_c -1$. 
\end{proof}
\begin{remark}\label{M_c}
    Observe that by the above proposition, if $\ggo_M$ admits a symplectic structure, then there exists a unique $c  \in \text{spec}_\RR (M)$ such that
    \begin{enumerate}
        \item[(I)] if $c=0$, then $m_0$ is odd and $M$ is conjugate to either $M(0)$ or $M(0)\oplus Q$ where    $0\not\in \text{spec}_\RR (Q)$, depending  on whether $m_0=2n-1$ or $m_0 < 2n-1$;  
        \item[(II)] if $c\neq 0$ and $m_c>1$, then $-c \in \text{spec}_\RR (M)$,  $m_{-c}=m_c -1 \geq 1$ and $M$ is conjugate to either $M(c)\oplus M(-c)$ or $M(c)\oplus M(-c) \oplus Q$ where $\pm c\not\in \text{spec}_\RR  (Q)$, depending on whether $m_c=n$ or $m_c<n$. In particular, $M(c)\oplus M(-c) \in M(2m_c -1, \RR)$;
        \item[(III)] if $c\neq 0$ and  $m_c=1$, then $M$ is conjugate to $M(c)\oplus Q$ where $c\not\in \text{spec}_\RR (Q)$. 
     \end{enumerate}
\end{remark}

\medskip

The cases (I), (II) and (III) of the above remark will be studied in Theorem \ref{sympM(c)+M(-c)} and Theorem \ref{sympM(c)+M(-c)+Q}. 

\begin{theorem}\label{sympM(c)+M(-c)}
    Let $M\in M(2n-1,\RR)$, $n\geq 2$. Consider the following two cases:
\begin{enumerate}
   \item[\emph{(I)}] $M=M(0)$. In this case,   
     $\ggo_M$ admits a symplectic structure if and only if $\Lambda (M) \in \bigwedge^s_{2n-1}$.

\medskip
    
    \item[\emph{(II)}] $M=M(c)\oplus M(-c)$, with $c\neq 0$ and $m_c=n$. In this case we have two possibilities:
    
 \noindent  $\diamond$ If 
    $ \Lambda (M(-c))= (n-1), $
    then $\ggo_M$ admits a symplectic structure if and only if \[ \Lambda (M(c))=\begin{cases} (n),  &\text{ if $M$ is diagonalizable}  \\
    \text{ or }\\
       (2;1;n-2) & \text{ if $M$ is not diagonalizable}
    \end{cases} \]
$\diamond$ If  $\Lambda (M(-c))=(n_1, \cdots, n_k; p_1, \cdots, p_k; t)   ,$   
    then $\ggo_M$ admits a symplectic structure if and only if one of the following conditions holds:
    \begin{enumerate}
        \item[\emph{(i)}] $\Lambda (M(c))=(n_1, \cdots, n_k; p_1, \cdots, p_k; t+1)$;
        \item[\emph{(ii)}] there exists a unique $d\in\{1, \cdots, k\}$ such that $\Lambda (M(c))$ is given by
        \[
        \begin{cases}
            (n_1, \cdots, n_{d-1}, n_d, \cdots,  n_k; p_1, \cdots,  p_{d-1}+1, p_{d}-1, \cdots,p_k; t), & \text{if } n_{d-1}=n_d +1, \\
       \text{ or } &\\
     (n_1, \cdots, n_{d-1}, n_d +1, n_d, \cdots,  n_k; p_1, \cdots,  p_{d-1}, 1, p_{d}-1, \cdots,p_k; t),  & \text{if }n_{d-1}>n_d + 1 > n_{d};
      \end{cases}
      \]
      \item[\emph{(iii)}] $ \Lambda (M(c))=\begin{cases}
          (n_1, \cdots, n_k, 2; p_1, \cdots, p_k, 1; t-1),  &\text{if } n_k >2,\\ 
      \text{or}&  \\
      (n_1, \cdots, n_k; p_1, \cdots, p_k + 1; t-1), & \text{if } n_k =2.
      \end{cases}$       
     \end{enumerate}
\end{enumerate}
\end{theorem}

\begin{proof}
Case (I) is just Theorem \ref{classifsymp}.

\noindent Case (II): Suppose that $\ggo_M$ admits a symplectic structure, then by Remark  \ref{M_c} and Corollary \ref{symp-v columna}, $M$ is conjugate to $C= \begin{pmatrix}
      c   & 0 \\
      w   & D
     \end{pmatrix}$ where $w\in \RR^{2(n-1)}$ and $D\in \spg(2(n-1), \RR)$. Therefore, by Theorem \ref{prop:symplex-nonzero},  $D$ satisfies condition  \eqref{symp-complex-eig} for $\lambda=c$. 
Since $m_c=n$ and $\text{spec}\, (M)=\{c, -c\}$, then $\text{spec}\, (D)=\{c, -c\}$. By Lemma \ref{B(a)+D}, $C$ is conjugate to $\left(\begin{array}{c|c|c}
      c  & 0  & 0\\
      \hline
      u &  D(c) & 0 \\
      \hline
      0 & 0 & D(-c)
      \end{array}\right)$, where $u\in \RR^{n-1}$.
Therefore, since $M$ is conjugate to the previous matrix and $D$ satisfies condition \eqref{symp-complex-eig} for $\lambda=c$, we obtain 
\begin{equation}\label{Lambda(M(-c))}
\Lambda (M (-c))= \Lambda (D (-c))=\Lambda (D (c)).
\end{equation}

First, we suppose that $\Lambda (M (-c))=(n-1)$, then the statement follows depending on whether $u=0$ or $u\neq 0$, which is equivalent to whether $M$ is diagonalizable or not. In fact, if $u=0$, then $M$ is conjugate to $c I_n \oplus (-c) I_{n-1}$ and $\Lambda(M(c))=(n)$; while if $u\neq 0$ then  by using a suitable change of basis it is easy to see that $\Lambda (M(c))=(2;1;n-2)$. 

Now, if $\Lambda (M(-c))=(n_1, \cdots, n_k; p_1, \cdots, p_k; t)$, we observe that $M(c)$ is conjugate to $\tilde{C}=\left(\begin{array}{cc}
      c  & 0  \\
      u &  D(c)       \end{array}\right)$. In particular, the non-zero nilpotent matrices $M(c)-cI_n$ and $\tilde{C}-cI_n$ are conjugate. 
      
Since, $\Lambda (M(c))= \Lambda (M(c)-cI_n)=\Lambda (\tilde{C}-cI_n)=\Lambda (\tilde{C})$,  it follows from \eqref{Lambda(M(-c))}   and Corollary  \ref{posibles_C} that $\Lambda (M(c))$ takes one of the forms (i), (ii) or (iii) of the statement.

Conversely, in all cases it can be shown that $M(c)$ is conjugate to $\left(\begin{array}{cc}
      c  & 0  \\
      u &  D_c       \end{array}\right)$ such that $\Lambda (D_c)=\Lambda (M(-c))$, for some $u\in \RR^{n-1}$.
Hence, $M$ is conjugate to $\left(\begin{array}{c|c|c}
      c  & 0  & 0\\
      \hline
      u &  D_c & 0 \\
      \hline
      0 & 0 & M(-c)
      \end{array}\right)$
where $D_c\oplus M(-c)\in \spg(2(n-1), \RR)$ by Theorem \ref{prop:symplex-nonzero}. Therefore, $\ggo_M$ admits a symplectic structure by Corollary \ref{symp-v columna}.
\end{proof}

\begin{remark}\label{pairstuples} 
    We point out that Theorem \ref{sympM(c)+M(-c)} (II) characterizes the matrices $M=M(c)\oplus M(-c)$ that admit a symplectic structure in terms of pairs of tuples $(\Lambda (M(-c)), \Lambda (M(c))) \in \Lambda _{n-1} \times \Lambda _{n}$, where $\Lambda (M(c))$ is determined by the tuple $\Lambda (M(-c))$ according to the possibilities given in (II).  
\end{remark}
Our next result generalizes Theorem 1.1 in \cite{CM}, which gives necessary and sufficient conditions for $\ggo_M$ to admit a symplectic structure, where $M\in M(2n-1, \RR )$ is a diagonal matrix. 

\begin{theorem}\label{sympM(c)+M(-c)+Q} Let $M\in M(2n-1, \RR )$, $n\geq 2$, $Q\in M(r, \RR)$, $r\geq 1,$  and  $c\in \emph{spec}_\RR (M)$ such that  $c\not\in \emph{spec}\,   (Q)$.
  \begin{enumerate}
   \item[\emph{(I)}] If $c=0$, $m_0>1$ is odd and
     $M=M(0) \oplus Q$, then $\ggo_M$ admits a symplectic structure if and only if 
     $\ggo_{M(0)}$ admits a symplectic structure and $Q\in \spg_{\Omega} (r,\RR)$, for some non-singular skew-symmetric $\Omega\in M(r, \RR)$.
     
\medskip
    
    \item[\emph{(II)}] Let $c\neq 0$ and $M=M_c \,\oplus Q$, where $M_c= M(c) \oplus M(-c)$,  $m_{-c}=m_c-1$ and  $- c\notin \emph{spec}\,(Q)$. Then $\ggo_M$ admits a symplectic structure if and only if $\ggo_{M_c}$ admits a symplectic structure and $Q\in \spg_{\Omega} (r,\RR)$, for some non-singular skew-symmetric $\Omega\in M(r, \RR)$.

\medskip
    \item[\emph{(III)}] If  $M=M(c)\oplus Q$  and $m_c=1$, then $\ggo_M$ admits a symplectic structure if and only if  $Q\in \spg_{\Omega} (r,\RR)$, for some non-singular skew-symmetric $\Omega\in M(r, \RR)$. 
   \end{enumerate}  
\end{theorem}

\begin{proof}
     We recall that if $\ggo_M$ admits a symplectic structure, then by Corollary \ref{symp-v columna}, $M$ is conjugate to 
     \begin{equation} \label{eq:symp-matrix2}
C= \begin{pmatrix}
      c   & 0 \\
      w   & D
\end{pmatrix}, \quad \text{where } c\in \RR , \,  w \in \mathbb R ^{2(m-1)}, \; D\in \spg(2(m-1), \RR).   
\end{equation}
     Therefore, by Theorem \ref{prop:symplex-nonzero},  $D$ satisfies conditions \eqref{symp-real-eig} and \eqref{symp-complex-eig}. We will use this notation several times throughout the proof.

     \smallskip

\noindent \textbf{Case (I):} Suppose that $c=0$, $m_0>1$ is odd and $M=M(0)\oplus Q$. Assume that $\ggo _M$ admits a symplectic structure.

Since $m_0>1$ we have that $0\in \text{spec}_\RR (D)$ and $D$ is conjugate to $D(0) \oplus Q'$ where $0\notin \text{spec} _\RR (Q')$. Then, Lemma~\ref{B(a)+D} implies that $C$ is conjugate to 
$$
C'=\left(\begin{array}{c|c|c}
      0  & 0  & 0\\
      \hline
      u &  D(0) & 0 \\
      \hline
      0 & 0 & Q'
      \end{array}\right),
$$      
where $u\in \RR ^{m_0 -1}$, $ D(0) \in \spg_{\Omega_1}(m_0 -1,\RR)$, and  $Q' \in \spg_{\Omega_2}(2n-1-m_0,\RR)$, since 
$D(0)$ and $Q'$ satisfy \eqref{symp-real-eig} and \eqref{symp-complex-eig}, respectively. In particular, $\ggo_{C_0}$ admits a symplectic structure for the non-zero nilpotent matrix $C_0:=\left(\begin{array}{cc}
      0  & 0  \\
      u &  D(0) \end{array}\right)$. In addition, since $C' = C_0\oplus Q'$ is conjugate to $M = M(0) \oplus Q$ it follows that $C_0$ is conjugate to $M(0)$ and $Q'$ is conjugate to $Q$. Finally, $\ggo_{M(0)}$ admits a symplectic structure and $Q\in \spg_{\Omega}(2n-1-m_0,\RR)$ for some  $\Omega$.

Conversely, suppose that $\ggo_{M(0)}$ admits a symplectic structure. In particular, by Corollary \ref{symp-v columna}, $M(0)$ is conjugate to a matrix $C_0 =\left(\begin{array}{cc}
      0  & 0  \\
      w' &  D'
      \end{array}\right)$, where $D'\in \spg (m_0-1, \RR)$ and $w'\in \RR ^{m_0 -1}$. By hypothesis, $Q\in \spg_{\Omega}( 2n-m_0 -1, \RR)$. Then, $\ggo_M$ admits a symplectic structure since $M$ is conjugate to $C_0 \oplus Q$, where $D'\oplus Q\in \spg_{\Omega'} (2n-1, \RR)$ with $\Omega'= J_{m_0-1}\oplus \Omega$ (see Remark \ref{suma w_i}). In this way, the equivalence given in Case (I) is completely proved.
      
\medskip
    
\noindent \textbf{Case (II):} Suppose $c\neq 0$, $M=M_c \,\oplus Q$, where $M_c= M(c) \oplus M(-c)$,  $m_{-c}=m_c-1$, and  $- c\notin \text{spec}\,(Q)$. 
If $\ggo _M$ admits a symplectic structure, then  $-c\in \text{spec} (D)$ with $D$ as in \eqref{eq:symp-matrix2}, since $-c\in \text{spec} (M)$.  By \eqref{Cdecomposition}, $D$ is conjugate to $D(c)\oplus D(-c) \oplus Q'$, where $\pm c \notin \text{spec}\,(Q')$. By Lemma \ref{B(a)+D}, $M$ is conjugate to
$$
\left(\begin{array}{c|c|c|c}
      c  & 0  & 0 & 0\\
      \hline
      u &  D(c) & 0 & 0\\
      \hline
      0 & 0 & D(-c) & 0\\
      \hline
      0 & 0 & 0 & Q'      
      \end{array}\right),
$$ 
and therefore $\Lambda (M(-c))=\Lambda (D(-c))$. 
In particular, $M_c=M(c) \oplus M(-c)$ is conjugate to 
$$
M'_c= \left(\begin{array}{c|c|c}
      c  & 0  & 0\\
      \hline
      u &  D(c) &  0\\
      \hline
      0 & 0 & D(-c) 
      \end{array}\right).
$$
and $Q$ is conjugate to $Q'$. Also, since $D\in \spg(2n-2,\RR)$, Theorem \ref{prop:symplex-nonzero} implies  that $\Lambda (D(c))=\Lambda (D(-c))$ and $Q'\in\spg_{\Omega'}(2m_c-2,\RR)$ for some $\Omega'$.  Then, $D(c)\oplus D(-c)\in \spg_{\Omega_1} (2m_c-2,\RR)$ and by Corollary \ref{symp-v columna} the Lie algebra $\ggo_{M'_c}$ admits a symplectic structure. Finally, the statement follows from the fact  that  $M_c$ and $Q$ are conjugate to $M_c'$ and $Q'$, respectively.

Conversely, if $\ggo_{M_c}$ admits a symplectic structure, then by Corollary \ref{symp-v columna}, $M_c$ is conjugate to $\left(\begin{array}{cc}
      c  & 0  \\
      u' &  D'
      \end{array}\right),$ where $D' \in \spg(2m_c-2,\RR)$. Then, $M$ is conjugate to $M_c \oplus Q$, where $Q\in \spg_{\Omega}(2n-2m_c,\RR)$ by hypothesis. Therefore, $M$ is conjugate to 
      $$
  \left(\begin{array}{c|c|c}
      c  & 0  & 0\\
      \hline
      u^{\prime} &  D' &  0\\
      \hline
      0 & 0 & Q 
      \end{array}\right),
$$    
where $D'\oplus Q\in  \spg_{\Omega'}(2n-2,\RR)$, where $\Omega'=J_{m_c -1}\oplus \Omega$, and Corollary \ref{symp-v columna} implies that $\ggo_M$ admits a symplectic structure.

\medskip

\noindent \textbf{Case (III):} Assume that $M=M(c)\oplus Q$  and $m_c=1$. If $\ggo _M$ admits a symplectic structure, since $Q$ is conjugate to $D$ with $D$ is as \eqref{eq:symp-matrix2}, then $Q\in \spg_{\Omega}(2n-2,\RR)$, for some $\Omega$. The converse in this case is straightforward by Corollary \ref{symp-v columna}.
\end{proof}
For $c\neq 0$ and $n\geq 2$, we denote by $\Sigma_{n,c}\subset \bigwedge _{n-1} \times \bigwedge _{n}$ the set of all pairs of tuples that we refer to in Remark \ref{pairstuples}.

Let $A\in M(2n-1, \RR)$ be a non-nilpotent matrix. By combining   Theorem \ref{sympM(c)+M(-c)}, Theorem \ref{sympM(c)+M(-c)+Q} and  Lemma \ref{lem:isom-classes} it follows that  
$\ggo_A$ admits a symplectic structure if and only if $A$ is conjugate to $b M$ for some  $b\in\RR, \, b\neq 0$, where either $1\in\text{spec}\,(M)$ or $0\in\text{spec}\,(M)$ and $M$ takes one of the  forms below:
\begin{enumerate}
 \item $M=M(0)\oplus Q$  with $m_0 \geq 1$ and odd,  $\Lambda (M(0))\in \bigwedge ^s_{m_0}$,  $Q\in\spg(2n-m_0-1, \RR)$ and $0\notin \text{spec}\,(Q)=\{ \mu_1, \ldots , \mu_s \}$ and $1=|\mu _1| \geq \cdots  \geq |\mu_s |>0$,
    \item  $M=M(1) \oplus Q$ where $m_1=1$,  $Q\in \spg(2n-2, \RR)$ and $1\notin \text{spec}\, (Q)$,
    \item $M=M(1)\oplus M(-1)$ where $m_1=m_{-1} +1$ and $(\Lambda (M (-1)), \Lambda (M (1)))\in \Sigma_{n,1}$, 
    \item $M=M(1)\oplus M(-1)\oplus Q$, where $(\Lambda (M (-1)), \Lambda (M (1)))\in \Sigma_{m_1,1}$ and  
    $Q\in \spg(2r, \RR)$ with $2r=2n-2-2m_{-1}$ and $\{1, -1 \}\not\subset \text {spec} \,(Q)$.
    \end{enumerate}\

We point out that, since $A$ is conjugate to $bM$ with $b\neq 0$,  it follows from Proposition \ref{ad-conjugated} that $\ggo_A$ is isomorphic to $\ggo_M$. Therefore, 
the isomorphism classes of $\ggo_A$ are parametrized by the isomorphism classes of $\ggo_M$. In order to obtain this parametrization, we  define the following quotient spaces:
\[ \mathcal S^1_r:=\{ Q\in \spg (2r, \RR) \,:\, 1 \notin \text {spec}\, (Q) \} /GL(2r,\RR ),  \]
\[   \mathcal{S}^0_r:= \{Q\in \spg (2r, \RR) : 0\notin \text {spec} (Q)  =\{ \mu_1, \ldots , \mu_s \} \text{ and } 1=|\mu _1| \geq \cdots  \geq |\mu_s |>0   \}/_{ \sim}
\]
where $Q_1\sim Q_2$ if and only if $Q_1$ is  $GL(2r,\RR )$-conjugate to $\pm Q_2$ (see Lemma \ref{lem:isom-classes}).

\begin{corollary} Let $A\in M(2n-1, \RR)$ be a non-nilpotent matrix. The isomorphism classes of $\ggo_A$  admitting a symplectic structure are
the  isomorphism classes of $\ggo_M$ for $M$ as in ${\rm (1)}$, ${\rm (2)}$, ${\rm (3)}$ or ${\rm (4)}$  above, which are parametrized, respectively, by:   
\begin{equation*}
   \text{${\rm (1)}$ } \;  \bigcup_{k=1}^{n-1} \, \left({\bigwedge} ^s_{{2k-1}}\times\; \mathcal S_{n-k}^0\right),
    \qquad \text{${\rm (2)}$ } \;  \mathcal S_{n-1}^1,  \qquad 
    \text{$\rm (3)$} \; \Sigma_{n,1},
    \qquad
    \text{${\rm (4)}$ } \;\bigcup_{k=2}^{n-1} \,\left( \Sigma_{k,1} \times\; \mathcal S_{n-k}^1\right).    
\end{equation*}
\end{corollary}

\section{Applications of Theorem   \ref{complex_M(a)+Q} and  Theorem \ref{sympM(c)+M(-c)+Q}}

We start by giving families of almost abelian Lie algebras admitting complex or symplectic structures, or neither complex nor symplectic structure. 
\begin{example}
 Let $M\in M(2n-1, \RR )$ such that   $M=M(1) \oplus Q$ with $m_1>1$ , $m_1$ odd,  $1 \not\in \text{spec}\, (Q)$ and 
 $-1 \not\in \text{spec}\, (M)$. If the conditions in Theorem \ref{complex_M(a)+Q} are satisfied, then $\ggo_M$ admits a complex structure but does not admit a symplectic structure. These shows that Theorem \ref{complex-sympl} does not hold in the general case, that is when $\ggo_M$ is an almost abelian solvable Lie algebra which is not nilpotent.  
\end{example}

\begin{example}
 Let $M\in M(2n-1, \RR )$ such that   $M=M(1) \oplus Q$ with $m_1>1$, $m_1$ odd, $1 \not\in \text{spec}\, (Q)$ and 
 $-1 \not\in \text{spec}\, (M)$. If any of the conditions in Theorem \ref{complex_M(a)+Q} are not satisfied then $\ggo_M$ admits neither
complex nor symplectic structures.    
\end{example}

\begin{example}
Let $\ggo_M$ be an almost abelian Lie algebra.
\begin{itemize}
    \item If $M=\J_{m+1}(1) \oplus \J_m(1) \oplus \J_{m}(-1) \oplus \J_m(-1)$, then $\ggo_M$ admits both, complex and symplectic structures for $m \geq 2$.
    \item If $M=\J_m(1)$ with $m$ odd, then $\ggo_M$ admits neither complex nor symplectic structures. 
    \item If $M=\J_{m+1}(1) \oplus \J_m(1) \oplus (-1) I_{t}$, then $\ggo_M$ admits complex but no symplectic structure for $m \geq 2$, $t$ even.  
    \item If $M=\J_{m+1}(1) \oplus \J_m(-1)$, then $\ggo_M$ admits symplectic but no complex structure for $m \geq 2$. 
\end{itemize}    
\end{example}
As a consequence of Theorem  \ref{complex_M(a)+Q} and Theorem \ref{sympM(c)+M(-c)+Q} we show that if $\ggo_M$ admits a complex or symplectic structure, then so do the almost abelian Lie algebras associated to the semisimple and nilpotent parts of $M$.
\begin{theorem}\label{teo:nilpotent-part}
   Let $\ggo_M$ be a $2n$-dimensional almost abelian Lie algebra and consider the Jordan-Chevalley decomposition of $M$, $M=M_s+M_n$ where $M_s$ is semisimple and $M_n$ is nilpotent. Assume that $\ggo_M$ admits a complex structure \emph{(}resp., symplectic structure\emph{)}, then $\ggo_{M_s}$ and $\ggo_{M_n}$ admit a complex structure \emph{(}resp., symplectic structure\emph{)}. 
\end{theorem}

\begin{proof} Since for $M=M_n$ or $M=M_s$ the statement trivially holds, we suppose that $M_s\neq 0_{2n-1}$ and $M_n\neq 0_{2n-1}$. 

If $\ggo_M$ admits a complex structure, then, by Proposition \ref{odd-space}, we may assume that $M=M(a)$ or $M=M(a)\oplus Q$ with $a\in\RR$, $m_a$ odd and $a\notin \text{spec}\,(Q)$. By Theorem \ref{complex_M(a)+Q}, $\Lambda (M(a))\in {\bigwedge}^c_{m_a}$ and $Q\in M_{J}(m,\RR)$, for some $J$ such that $J^2=-I_m$ with $m=2n-1-m_a$.

 We notice that \begin{itemize} 
 \item if $M=M(a)$, then $M_s=aI_{m_a}$ and $M_n= (M(a)-a I_{m_a})$,  
 \item if $M=M(a)\oplus Q$, then  $M_s=aI_{m_a}\oplus Q_s$ and $M_n= (M(a)-a I_{m_a})\oplus Q_n$ (taking $Q_s=0$ when $Q$ is a nilpotent matrix and $Q_n =0$ when $Q$ is a semisimple matrix).
 \end{itemize}
 Since $\Lambda (M(a)-a I_{m_a})=\Lambda (M(a))\in \bigwedge^c _{m_a}$ and $\Lambda (a I_{m_a})\in \bigwedge^c _{m_a}$, 
 %because $m_a$ is odd, 
 then from Theorem \ref{complex_M(a)+Q}, Theorem \ref{classifcomplex} and Proposition \ref{solv-complex2n}, we obtain that $\ggo_{M_s}$ and $\ggo_{M_n}$ admit a complex structure. 

Now, if $\ggo_M$ admits a symplectic structure, then, by Remark \ref{M_c}, we may assume that $M$ takes one of the following forms: 
\begin{enumerate}
        \item[(i)] $M=M(0)\oplus Q$  where $m_0>1$ is odd, $m_0 \leq 2n-3$ and $0\not\in \text{spec}_\RR (Q)$,  
        \item[(ii)] $M=M(c)\oplus M(-c)$ or $M=M(c)\oplus M(-c) \oplus Q$ where $m_{-c}=m_c -1 \geq 1$ and $\pm c\not\in \text{spec}_\RR  (Q)$,
        \item[(iii)] $M=M(c)\oplus Q$ where $m_c=1$ and $c\not\in \text{spec}_\RR (Q)$. 
     \end{enumerate}

If $M$ is as in (i), by Theorem \ref{sympM(c)+M(-c)+Q} (I) we have $\Lambda (M(0))\in \bigwedge^s_{m_0}$, $0\notin \text {spec } (Q)$ and $Q\in \spg_{\Omega}(2r, \RR)$ for some $\Omega$. Then, $M_s=0_{m_0} \oplus Q_s$ and $M_n=M(0)\oplus Q_n$. Theorem \ref{sympM(c)+M(-c)+Q} (I), Theorem \ref{classifsymp} and \ref{prop:symplex-nonzero} imply that  $\ggo_{M_s}$ and $\ggo_{M_n}$ admit a symplectic structure.  

If $M$ is as in (ii), by  Theorem \ref{sympM(c)+M(-c)} (II) and Theorem \ref{sympM(c)+M(-c)+Q} (II) we have $(\Lambda (M(-c)), \Lambda (M(c))) \in \Sigma _{m_c,c}$ and $Q\in \spg_{\Omega}(2r, \RR)$ for some $\Omega$. Then 
$$
M_s=cI_{m_c} \oplus (-c)I_{m_{-c}}\,\,\, \text{and} \,\,\,M_n=(M(c)-cI_{m_c}) \oplus \left(M(-c)-(-c)I_{m_{-c}}\right),
$$ 
or 
$$
M_s=cI_{m_c} \oplus (-c)I_{m_{-c}} \oplus Q_s \,\,\, \text{and}\, \,\,M_n=(M(c)-cI_{m_c}) \oplus \left(M(-c)-(-c)I_{m_{-c}}\right) \oplus Q_n,
$$
respectively.
Hence, $(\Lambda (-cI_{m_{-c}}), \Lambda (cI_{m_c})) \in \Sigma _{m_c,c}$ and 
\begin{equation} \label{eq:sympss}
(\Lambda (M(-c)+c I_{m_{-c}}), \Lambda (M(c)-cI_{m_{c}}))=(\Lambda (M(-c)), \Lambda (M(c))) \in \Sigma _{m_c,c}.   
\end{equation}
Using \eqref{eq:sympss} it can be shown that
$\Lambda((M(c)-cI_{m_c}) \oplus \left(M(-c)-(-c)I_{m_{-c}}\right))\in \bigwedge^s_{2m_c-1}.$  Theorem \ref{sympM(c)+M(-c)+Q} (II), Theorem \ref{classifsymp} and \ref{prop:symplex-nonzero} imply that  $\ggo_{M_s}$ and $\ggo_{M_n}$ admit a symplectic structure.

Finally, if $M$ is as in (iii), by  Theorem \ref{sympM(c)+M(-c)+Q} (III), we have that $Q\in \spg_{\Omega}(2n-2, \RR)$ for some $\Omega$. Since  $M_s=M(c)\oplus Q_s$ and $M_n=0_1\oplus Q_n$, 
Theorem \ref{prop:symplex-nonzero} and Corollary \ref{symp-v columna} imply that $\ggo_{M_s}$ and $\ggo_{M_n}$ admit a symplectic structure.
\end{proof}
\begin{example} 
The converse of Theorem \ref{teo:nilpotent-part}  does not hold. In fact, let  $b\in \RR$, $b\neq 1$  and let $m$ and  $\ell$ be distinct odd positive integers.  Consider 
\[ M= 0_1\oplus \J_{m}(1)\oplus \J_{\ell}(1) \oplus \J_{m}(b) \oplus\J_{\ell}(b) ,\] then the semisimple and nilpotent parts of $M$ are given by 
\[ M_s=0_1 \oplus  I_{m+\ell } \oplus b I_{m+\ell } , \quad \text{and}\quad  \quad M_n=\J_m ^{\oplus 2} \oplus \J_\ell  ^{\oplus 2} \oplus 0_1 ,\]
 respectively. By Theorem   \ref{complex_M(a)+Q}, it follows that $\ggo _M$  does not admit any complex structure. However, both $\ggo _{M_{n}}$ and $\ggo _{M_{s}}$ admit complex structures.

Now,  let $k$ and $t$ be even and odd positive integers, respectively. Consider 
\[
M=I_{k+1}\oplus (-1)I_k \oplus \J_{k+1}\oplus 0_t,
\]
then the semisimple and nilpotent parts of $M$ are given by 
\[ M_s= I_{k+1} \oplus (-1) I_{k} \oplus 0_{k+1+t}, \quad\text{and} \quad  \quad M_n=\J_{k+1} \oplus  0_{2k+1+t} ,\]
respectively. By Theorem \ref{sympM(c)+M(-c)+Q}, it follows that $\ggo _M$ does not admit any symplectic structure. However, both $\ggo _{M_{n}}$ and $\ggo _{M_{s}}$ admit symplectic structures.
\end{example}


\begin{thebibliography}{99}
\bibitem{AV}
A. Abasheva, M. Verbitsky, {\it 
Algebraic dimension and complex subvarieties of hypercomplex nilmanifolds}, 
Adv. Math. {\bf 414} (2023), Article  108866.

\bibitem{AB}
A. Andrada, M.L. Barberis, {\it Applications of the quaternionic Jordan form to hypercomplex geometry}, J. Algebra \textbf{664}, Part B (2025), 73--122.

%\bibitem{AD}
%A . Andrada, I. Dotti, {\it Killing-Yano 2-forms on 2-step nilpotent Lie groups},  
%Geom. Dedicata \textbf{212}  (2021), 415--424.


\bibitem{AO}
A. Andrada, M. Origlia, {\it Lattices in almost abelian Lie groups with locally conformal K\"ahler or symplectic structures}, Manuscripta Math. \textbf{155} (2018), 389--417.

\bibitem{AN}
R.M. Arroyo, M. Nicolini, {\it SKT structures on nilmanifolds}, { Math. Z.} 
\textbf{
 302} (2022), 1307--1320.

\bibitem{Baz}
G. Bazzoni, {\it 
Vaisman nilmanifolds},  
Bull. Lond. Math. Soc. {\bf 49} (2017), 824--830.



\bibitem{BFLT}
G. Bazzoni, M. Freibert, A. Latorre, N. Tardini, 
{\it Complex symplectic Lie algebras with large abelian subalgebras}, 
Linear Algebra Appl. {\bf 677} (2023), 254--305.

\bibitem{BGM}
G. Bazzoni, A. Garvín, V.  Muñoz, {\it 
Purely coclosed $G_2$-structures on nilmanifolds}, 
Math. Nachr. {\bf 296} (2023), 2236--2257.

\bibitem{BG}
Ch. Benson, C. S. Gordon, {\it K\"ahler and symplectic
structures on nilmanifolds}, Topology \textbf{27} (1988),  513--518.

\bibitem{Ca}
R. Campoamor-Stursberg, {\it Symplectic forms on six-dimensional real solvable Lie algebras
I}, Algebra Colloq. {\bf 16} (2009), 253–266.


\bibitem{CM}
L.P. Castellanos Moscoso, {\it Left-invariant symplectic structures on diagonal almost abelian Lie groups}, 
Hiroshima Math. J. {\bf 52} (2022),  357--378.

\bibitem{CT}
L.P. Castellanos Moscoso, H. Tamaru, 
{\it A classification of left-invariant symplectic structures on some Lie groups}, 
Beitr. Algebra Geom. {\bf 64} (2023), 471--491.

\bibitem {CG}
G.R. Cavalcanti, M. Gualtieri, {\it Generalized complex structures
on nilmanifolds},  J. Symplectic Geom. {\bf 2}  (2004), 393--410.

\bibitem{COUV16}
M. Ceballos, A. Otal, L. Ugarte, R. Villacampa,  {\it Invariant complex structures on 6-nilmanifolds: classification, Frölicher spectral sequence and special Hermitian metrics}, J. Geom. Anal. \textbf{26} (2016), no. 1, 252--286.

\bibitem{ChH}
N. Chaemjumrus, C.M. Hull, 
{\it The doubled geometry of nilmanifold reductions}, 
J. High Energy Phys. {\bf 12} (2019),  Paper No. 157.

%\bibitem{Chu}
%B. Y. Chu,  {\it Symplectic homogeneous spaces}, Trans. Amer. Math. Soc. {\bf 197} (1974), 145–159.

\bibitem {CMcG}
D. Collingwood, W. McGovern, Nilpotent orbits in semisimple
Lie algebras, Van Nostrand Reinhold, New York,  
1993.

\bibitem{CFG}
L.A. Cordero, M. Fernández, A. Gray, 
{\it Symplectic manifolds with no K\"ahler structure}, 
Topology \textbf{25} (1986), 375--380.

\bibitem{dB0}
V. del Barco, A. Moroianu, {\it Killing forms on 2-step nilmanifolds}, J. Geom. Anal. \textbf{31} (2021), 863--887. 

\bibitem{dB1} 
V. del Barco, {\it Symplectic structures on nilmanifolds: an obstruction for their existence}, J. Lie Theory \textbf{24} (2014), 889--908.  

\bibitem{dB2}
V. del Barco, {\it Symplectic structures on free nilpotent Lie algebras}, Proc. Japan Acad., Ser. A \textbf{95} (2019), 88--90.

 %\bibitem{DLV}
%A.J. Di Scala, J. Lauret, L. Vezzoni,  {\it Quasi-Kähler Chern-flat manifolds and complex 2-step nilpotent
%Lie algebras}, Ann. Sc. Norm. Super. Pisa Cl. Sci. (5) \textbf{11} (2012), 41--60.

\bibitem{DT}
I. Dotti, P. Tirao, {\it Symplectic structures on Heisenberg-type nilmanifolds},  
Manuscr. Math. \textbf{102} (2000),  383--401 (2000).

\bibitem {FGG}
M.  Fern\'{a}ndez, M. Gotay, A. Gray, 
{\it  Compact parallelizable four dimensional symplectic and complex
 manifolds}, Proc. Amer. Math. Soc. {\bf 103} (1988), 1209--1212.

\bibitem{FP1}
A. Fino, F. Paradiso, {\it Balanced Hermitian structures on almost abelian Lie algebras}, J. Pure Appl. Algebra {\bf 227} (2023), Article 107186.


\bibitem{Fr}
M. Freibert, {\it Cocalibrated structures on Lie algebras with a codimension one Abelian ideal}, {Ann. Glob. Anal.Geom.} \textbf{42} (2012), 537--563. 

\bibitem{GZZ}
Q. Gao,  Q. Zhao,  F. Zheng, {\it Maximal nilpotent complex structures},   Transform. Groups \textbf{28} (2023), 241--284. 

\bibitem{GJK}
J. R. G\'omez, A. Jim\'enez-Merch\'en,  Y. Khakimdjanov, {\it Symplectic structures on the filiform
Lie algebras}, J. Pure Appl. Algebra {\bf 156} (2001), 15–31.

\bibitem{GB}
M. Goze, A. Bouyakoub: {\it Sur les alg\`ebres de Lie munies d’une forme symplectique}, Rend. Sem. Fac. Sci. Univ. Cagliari {\bf 57} (1) (1987) 85–97.


\bibitem{GK}
M. Goze, Y. Khakimdjanov, {\it Nilpotent Lie algebras}, Mathematics and its Applications,
361. Kluwer Academic Publishers Group, Dordrecht, 1996.

\bibitem{GR}
M. Goze, E. Remm, 
{\it Non existence of complex structures on filiform Lie algebras}, 
Commun. Algebra \textbf{30} (2002),  3777--3788.


\bibitem{Gua}
Z.-D. Guan, {\it Toward a classification of compact nilmanifolds with
symplectic structures}, Int. Math. Res. Not. IMRN \textbf{2010} (2010), 4377--4384.

%\bibitem{Iwa}
%K. Iwasawa, {\it On some types of topological groups}, Ann. of Math. {\bf 50} (1949), 507--558. 

%\bibitem{Ja}
%J. C. Jantzen, Nilpotent orbits in representation theory, Lie theory, Progr. Math., vol. 228,
%Birkhäuser Boston, Boston, MA, 2004.

\bibitem{LUV19}
A. Latorre, L. Ugarte, R. Villacampa, {\it The ascending central series of nilpotent Lie algebras with complex structure}, Trans. Amer. Math. Soc. \textbf{372} (2019), no. 6, 3867--3903.

\bibitem{LUV23} 
A. Latorre, L. Ugarte, R. Villacampa, {\it Complex structures on nilpotent Lie algebras with one-dimensional center}, J. Algebra \textbf{614} (2023), 271--306.


\bibitem{LRV}
J. Lauret, E. A. Rodr\'\i guez Valencia, {\it On the Chern-Ricci flow and its solitons for Lie groups}, {Math. Nachr.} \textbf{288} (2015), 1512--1526. 

\bibitem{LW}
J. Lauret, C. Will, {\it On the symplectic curvature flow for
locally homogeneous manifolds}, J. Symplectic Geom. {\bf  15} (2017), 1--49.

\bibitem{LZ}
Y. Li, F. Zheng, {\it Complex nilmanifolds with constant holomorphic sectional curvature}, Proc. Am. Math. Soc. \textbf{150} (2022), 319--326.

%\bibitem {Mag}
%L. Magnin,
%{\it  Sur les algebres de Lie nilpotentes de dimension $\leq 7$}, J. Geom. Phys.
% {\bf 3} (1986), 119--131.

%\bibitem{Mal}
%A. Malcev, On a class of homogeneous spaces, \textit{Izv. Akad. Nauk SSSR} \textbf{ 13} (1949),  9--32; English translation in \textit{Am. Math. Soc. Transl.} \textbf{ 39} (1951).

\bibitem{Mehl}
C. Mehl, V. Mehrmann, A. C. M. Ran,  L. Rodman, 
{\it Eigenvalue perturbation theory of structured real matrices and their sign characteristics under generic structured rank-one perturbations}, 
Linear Multilinear Algebra {\bf 64} (2016), 527--556.

%\bibitem{Mi} 
%J. Milnor, {\it Curvatures of left invariant metrics on Lie groups}, {Adv. Math.} \textbf{21} (1976), 293--329.

\bibitem{NW}
Y. Nikolayevsky, J.A. Wolf, Joseph A., 
{\it The structure of geodesic orbit Lorentz nilmanifolds}, 
J. Geom. Anal. {\bf 33} (2023),  Paper No. 82.

%\bibitem{Nom}
%K. Nomizu, {\it On the cohomology of compact homogeneous spaces of nilpotent Lie groups}, {Ann. of Math.} \textbf{59} (1954), 531--538.

\bibitem{OOS}
L. Ornea, A.-I. Otiman, M.  Stanciu, 
{\it Compatibility between non-Kähler structures on complex (nil)manifolds}, 
Transform. Groups {\bf 28}  (2023), 1669--1686.

\bibitem{Ov}
G. Ovando, 
{\it Invariant complex structures on solvable real Lie groups}, 
Manuscr. Math. {\bf 103} (2000), 19--30.

\bibitem{Ov2}
G. Ovando, {\it Four dimensional symplectic Lie algebras}. Beitr\"age Algebra Geom. {\bf 47} (2006),
419–434.

\bibitem {OS}
G.P. Ovando, M.  Subils, 
{\it Symplectic structures on low dimensional 2-step nilmanifolds}, 
Preprint 2022, arXiv:2211.05768. 

\bibitem{Rol}
S. Rollenske, {\it Geometry of nilmanifolds with left-invariant complex structure and deformations in the large}, Proc.
London Math. Soc. \textbf{99} (2009),  425--460.

\bibitem {Sal}
S. Salamon, {\it Complex structures on nilpotent Lie algebras},
 J. Pure Applied Algebra  {\bf 157} (2001), 311--333.

\bibitem{Th}
 W.P. Thurston, {\it Some simple examples of symplectic manifolds}, Proc. Am. Math. Soc. \textbf{55} (1976), 467--468.

%\bibitem{Tr}
% A. Tralle, W. Kedra, {\it Compact completely solvable K\"ahler solvmanifolds are tori}, Int. Math. Res. Not. \textbf{15} (1997), 727--732.

\bibitem{U}
L. Ugarte, {\it Hermitian structures on six-dimensional nilmanifolds}, Transform. Groups \textbf{12} (2007), 175--202 . 


%\bibitem {Terr}
 %A. Terras, {\it Harmonic Analysis on Symmetric Spaces and Applications II}, Springer-Verlag,
%Berlin, 1988.

\end{thebibliography}
\end{document}